\documentclass{amsart}
\usepackage{amsmath, amsthm, amssymb}
\usepackage{verbatim}

\usepackage{url}

\renewcommand{\d}{\partial}
\newcommand{\ddbar}{\sqrt{-1}\d\overline{\d}}
\newcommand{\ii}{\sqrt{-1}}

\newcommand{\vepsilon}{\varepsilon}
\newcommand{\vphi}{\varphi}

\newcommand{\cA}{\mathcal{A}}

\newcommand{\cD}{\mathcal{D}}
\newcommand{\cE}{\mathcal{E}}

\newcommand{\cF}{\mathcal{F}}

\newcommand{\cS}{\mathcal{S}}
\newcommand{\cO}{\mathcal{O}}

\newcommand{\bR}{\mathbb{R}}
\newcommand{\bC}{\mathbb{C}}

\newtheorem{thm}{Theorem}
\newtheorem{prop}[thm]{Proposition}
\newtheorem{lem}[thm]{Lemma}
\newtheorem{cor}[thm]{Corollary}

\theoremstyle{definition}
\newtheorem{defn}[thm]{Definition}
\newtheorem{remark}[thm]{Remark}
\newtheorem{claim}[thm]{Claim}

\numberwithin{thm}{section}
\numberwithin{equation}{section}

\renewcommand{\[}{\begin{equation}}
\renewcommand{\]}{\end{equation}}

\newcommand{\wed}{\wedge}
\newcommand{\ga}{\gamma}

\title[Hessian equations on Hermitian manifolds with boundary]{The Dirichlet problem for a complex Hessian equation on compact Hermitian manifolds with boundary}

\author[D. Gu and N.-C Nguyen]{Dongwei Gu and Ngoc Cuong Nguyen}
\address{Faculty of Mathematics and Computer Science, Jagiellonian University 30-348 Krak\'ow, \L ojasiewicza 6, Poland}
\email{dongwei.gu@im.uj.edu.pl, \quad Nguyen.Ngoc.Cuong@im.uj.edu.pl}

\address{Department of Mathematics and Center for Geometry and its Applications, Pohang University of Science and Technology, 37673, The Republic of Korea}
\email{cuongnn@postech.ac.kr}


\begin{document} 

\begin{abstract} We solve the classical Dirichlet problem for a general complex Hessian equation on a small ball in $\bC^n$. Then,  we show that there is a continuous solution, in pluripotential theory sense, to the Dirichlet problem on compact Hermitian manifolds with boundary that equipped locally conformal K\"ahler metrics, provided a subsolution. \end{abstract}

\maketitle

\section{Introduction} \label{sec-intr}

Let $(\bar{M}, \alpha)$ be a compact Hermitian manifold with smooth boundary $\d M$, of complex dimension $n$. Let us denote $M := \bar M \setminus \d M$. Let $1\leq m \leq n$ be an integer. Fix a real $(1,1)$-form $\chi$ on $\bar M$. We have given a right hand side $f \in C^\infty(\bar M)$ positive and a smooth boundary data $\vphi \in C^\infty(\d M)$. The classical Dirichlet problem for the complex Hessian equation is to find a real-valued function $u \in C^\infty(\bar M)$:
\[ \label{intro-dp} \begin{aligned} 
&(\chi + dd^c u)^m \wedge \alpha^{n-m} = f \alpha^{n}, \\
&u = \vphi \quad \mbox{on } \d M,
\end{aligned}
\]
where $u$ is subjected to point-wise inequalities
\[\label{intro-m-positive}
	 (\chi + dd^c u)^k \wedge \alpha^{n-k}>0,\quad  k=1,..,m.
\]
We first solve the equation in a small ball.

\begin{thm} \label{thm:intr-1}
Let $M= B(z,\delta) \subset\subset B(0,1)$ be a Euclidean ball of radius $\delta$ in the unit ball $B(0,1)\subset \bC^n$. Assume that $\chi, \alpha$ are smooth on $\overline{B(0,1)}$. Then, the classical Dirichlet problem~\eqref{intro-dp} is uniquely solvable for $\delta$  small enough, which depends only on $\chi, \alpha$.
\end{thm}

A $C^2$ real-valued function satisfying inequalities~\eqref{intro-m-positive} is called $(\chi,m)-\alpha$- subharmonic. These inequalities can be generalised to non-smooth functions to obtain the class of $(\chi,m)-\alpha$-subharmonic functions on  $M$. Locally, the convolution of a function in this class with a smooth kernel, in general, will not belong to this class again. However, using the theorem above and an adapted potential theory, we prove the approximation property.

\begin{cor} \label{cor:intro-2} Any $(\chi,m)-\alpha$-subharmonic function  on $M$ is locally approximated by a decreasing sequence of smooth $(\chi,m)-\alpha$-subharmonic functions.
\end{cor}


Following Bedford-Taylor \cite{BT76,BT82,BT87} and Ko\l odziej \cite{kol98, kol03, kol05},  the two results above allow us to use Perron's envelope together with pluripotential theory techniques, adapted to this setting, to study weak solutions to this equation with the continuous right hand sides. A Hermitian metric $\alpha$ is called a locally conformal K\"ahler metric on $M$ if at every given point on $M$, there exist a local chart $\Omega$ and a smooth real-valued function $G$ such that $e^G \alpha$ is K\"ahler on $\Omega$. This class of metric is strictly larger than the K\"ahler one, and not every Hermitian metric is locally conformal K\"aher (see e.g. \cite{bru10}). Our main result is 

\begin{thm}\label{thm:intro-3} Assume that $\alpha$ is locally conformal K\"ahler. Let $0\leq f\in C^0(\bar M)$ and $\vphi \in C^0(\d M)$.  Assume that there is a $C^2$-subsolution $\rho$, i.e.,  $\rho$ satisfying inequalities \eqref{intro-m-positive} and
\[ \notag
	(\chi + dd^c \rho)^m \wedge \alpha^{n-m} \geq f \alpha^{n}\quad\mbox{in } \bar M, \quad
	\rho = \vphi \quad\mbox{on } \d M.
\]
Then, there exists a continuous solution to the Dirichlet problem \eqref{intro-dp} in pluripotential theory sense.
\end{thm}

When $m=n$ we need not assume $\alpha$ is locally conformal K\"ahler.  The Dirichlet problem for the Monge-Amp\`ere equation on compact Hermitian manifolds with boundary has been studied extensively, in smooth category, in recent years by Cherrier-Hanani \cite{cherrier-hanani98, cherrier99}, Guan-Li \cite{GL10} and Guan-Sun \cite{GS13}. Our theorem generalises the result in \cite{GL10} for continuous datum. 

When $1<m<n$ and $\alpha = dd^c |z|^2$ is the Euclidean metric the Dirichlet problem for the complex Hessian equation in a domain in $\bC^n$ has been studied extensively by many authors \cite{blocki05, chab, dk14, Li04, chinh13a, chinh15, cuong-thesis}. To our best knowledge the classical Dirichlet problem \eqref{intro-dp} on a compact Hermitian (or K\"ahler) manifold with boundary still remains open. The difficulty lies in the $C^1-$estimate for a general Hermitian metric $\alpha$. Here we only obtain such an estimate in a small ball (Theorem~\ref{thm:intr-1}). Moreover, in our approach, the locally conformal K\"ahler assumption of $\alpha$ is needed to define the complex Hessian operator of bounded functions (Section~\ref{sec:2}).  

Motivations to study the Dirichlet problem for such equations come from recent developments of fully non-linear elliptic equations on compact complex manifolds. First, it is the natural problem after the complex Hessian equation was solved by Dinew-Ko\l odziej \cite{dk15} on compact K\"ahler manifolds, and by Sz\'ekelyhidi \cite{szekelyhidi15} and Zhang \cite{dzhang15} on compact Hermitian manifolds. Indeed, such a question is raised in \cite{szekelyhidi15}. Second, on compact Hermitian manifolds, it is strongly related to the elementary symmetric positive cone with which several types of equations associated were studied by Sz\'ekelyhidi-Tosatti-Weinkove \cite{SzTW}, Tosatti-Weinkove \cite{TW13a, TW13b}. Our results may provide some tools to study these cones. In the case when $\alpha$ is K\"ahler ($\chi$ may be not), the Hessian type equations related to a Strominger system, which generalised Fu-Yau equations \cite{FY08}, have been studied recently by Phong-Picard-Zhang \cite{PPZ, PPZ15b, PPZ16}. Lastly, the viscosity solutions of fully nonlinear elliptic equations on Riemannian and Hermitian manifolds have been also investigated by Harvey and Lawson \cite{HLa, HLb} in a more general frame work, and the existence of continuous solutions was proved under additional assumptions on the relation of the group structure of manifolds and given equations. 

\medskip

{\em Organisation.} In Section~2 we give definitions for generalised $m-$ subharmonic functions and their basic properties. Assuming Theorem~\ref{thm:intr-1}, in Section~\ref{sec:2} we develop ``pluripotential theory" for corresponding generalised $m-$ subharmonic functions to the equation. This enable us to prove Corollary~\ref{cor:intro-2}. Section~\ref{sec:3} is devoted to study weak solutions to the Dirichlet problem in a small Euclidean ball. Theorem~\ref{thm:intro-3} is proved in Section~\ref{sec:bbd-env}. Finally, in Sections~\ref{sec:proof-main-thm}, \ref{sec:c0}, \ref{sec:c1}, \ref{sec:c2} we prove Theorem~\ref{thm:intr-1} independent of the other sections. The appendix is given in Section~\ref{sec:appendix}.

\medskip
 
{\em Acknowledgement.} We are grateful to S\l awomir Ko\l odziej for many valuable comments, which help to improve significantly the exposition of the paper. We also thank S\l awomir Dinew who has read the draft version of our paper and pointed out some mistakes. Furthermore, we would like to thank Blaine Lawson for bringing our attention to results in \cite{HLa, HLb, HLP}. Part of this work was written during the visit of the second author at Tsinghua Sanya International Mathematics Forum in January 2016. He would like to thank the organisers for the invitation and the members of the institute for their hospitality. The second author is supported by NCN grant 2013/08/A/ST1/00312.  The first author is supported by the Ideas Plus grant 0001/ID3/2014/63 of the Polish Ministry Of Science and Higher Education. 
Lastly, we would like to thank the referee for the helpful comments.

\section{Generalised $m$-subharmonic functions} \label{sec:1}

Fix a Hermitian metric $\alpha = \ii \alpha_{i\bar j} dz^i \wed d\bar z^j$ on a bounded open set $\Omega$ in $\bC^n$. Consider another real $(1,1)$-form $\chi = \ii \chi_{i\bar j} dz^i \wed d\bar z^j.$

A $C^2$ function $u$ on $\Omega$ is called $\alpha$-subharmonic if 
\[ \label{alpha-lap}
	\Delta_\alpha u(z) = \sum \alpha^{\bar j i}(z) \frac{\d^2 u }{\d z^i \d \bar z^j} (z) \geq 0,
\]
where $\alpha^{\bar j i}$ is the inverse of $\alpha_{i \bar j}$. We can rewrite it simply in  term of $(n,n)$-positive forms
\[ \notag
	dd^c u \wed \alpha^{n-1} \geq 0, \quad \mbox{where } dd^c = \frac{\ii}{\pi} \d \bar\d.
\]
This form has the advantage that one can generalise to non-smooth functions and with possibility define higher power of the wedge product of $dd^c u$ (see Remark~\ref{rem:higher-power-wed}). We start with the following definition which is adapted from subharmonic functions. 

\begin{defn} \label{def:alpha-sh}
A function $u: \Omega \to [-\infty,+\infty[$ is called $\alpha-$subharmonic  if 
\begin{enumerate}
\item[(a)]  $u$ is upper semicontinuous and $u\in L^1_{loc}(\Omega)$;
\item[(b)] for every relatively compact open set $D\subset\subset \Omega$ and every $h \in C^0(\bar D)$ and $\Delta_\alpha h =0$ in $D$, if $h \geq u$ on $\d D$, then $h \geq u$ on $\bar D$.
\end{enumerate}
\end{defn}

\begin{remark} Comparing to subharmonic functions we have that
\begin{enumerate}
\item
If an upper semicontinuous $u$ satisfies $(b)$, then by  Harvey-Lawson \cite[Theorem 9.3(A)]{hl13} it follows that  either $u\equiv -\infty$ or $u \in L^1_{loc}(\Omega)$. 
\item
The $\alpha-$subharmonicity for continuous function $u$ is equivalent to the inequality
$
	\Delta_\alpha u \geq 0
$
in the distributional sense, a detailed statement of this fact will be given in Lemma~\ref{lem:positive-dis-omega-sh} (Appendix).
\end{enumerate}
\end{remark}

We shall define $(\chi,m)-\alpha$-subharmonicity for non-smooth functions.
Let us denote 
\[ \notag
	\Gamma_m = \{\lambda= (\lambda_1,...,\lambda_n) \in \bR^n: S_1(\lambda)>0,...,S_m(\lambda)>0\}.
\]
The positive cone $\Gamma_m(\alpha)$ associated with the metric $\alpha$ is defined as follows. 
\[ 
	\Gamma_m(\alpha) = \{\gamma\mbox{ real } (1,1)-\mbox{form: } \gamma^k \wed \alpha^{n-k}>0 \mbox{ for every } k=1,...,m \}.
\]
In other words, in the orthonormal coordinate such that $\alpha = \sum_{i} \ii dz^i \wed d\bar z^i$ at a given point in $\Omega$, and $\gamma = \sum_{i} \lambda_i \ii dz^i \wed d\bar z^i$ also diagonalised at this point,  then  $\gamma \in \Gamma_m(\alpha)$ if  $(\lambda_1,...,\lambda_n)\in \Gamma_m.$

\begin{defn} \label{def:m-sub} A function $u: \Omega \to [-\infty, +\infty[$ is called $m-\alpha$-subharmonic if $u$ is $\tilde\alpha$-subharmonic for any $\tilde\alpha$ of the form 
$
	\tilde\alpha^{n-1}=\ga_1\wed\cdots \wed\ga_{m-1}\wed\alpha^{n-m},
$
where $\gamma_1, ...,\gamma_{m-1} \in \Gamma_m(\alpha)$.
\end{defn}

Here, the metric $\tilde \alpha$ is uniquely defined thanks to a result of Michelsohn \cite{michelsohn82}. By a simple consideration we have a generalisation 

\begin{defn} \label{def:ga-m-sub} A function $u: \Omega \to [-\infty, +\infty[$ is called $(\chi, m)-\alpha$ - subharmonic if $u + \rho$ is $\tilde\alpha$-subharmonic for any $\tilde\alpha$ of the form 
$
	\tilde\alpha^{n-1}=\ga_1\wed\cdots \wed\ga_{m-1}\wed\alpha^{n-m},
$
where $\gamma_1, ...,\gamma_{m-1} \in \Gamma_m(\alpha)$, and the  smooth function $\rho$ is defined, up to a constant, by the equation
$
	dd^c \rho \wed \tilde\alpha^{n-1}= \chi\wed\tilde\alpha^{n-1}. 
$
\end{defn}

Notice that when $\chi \equiv 0$, Definition~\ref{def:ga-m-sub} coincides with Definition~\ref{def:m-sub}. Thanks to Lemma~\ref{lem:positive-dis-omega-sh} in Appendix, we get that for a $(\chi, m)-\alpha$-subharmonic function $u$,
\[\label{eq:chi_u-m-positive}
 	(\chi+dd^c u) \wed \ga_1\wed\cdots \wed\ga_{m-1}\wed\alpha^{n-m} \geq 0
\]
for any collection $\ga_i\in \Gamma_m(\alpha),$ in the weak sense of currents. We denote the set of all $(\chi,m)-\alpha$-subharmonic functions in $\Omega$ by 
\[  \notag
SH_{\chi,m}(\alpha, \Omega) \quad \mbox{or} \quad 
SH_{\chi,m}(\alpha) \]
(for short) if the considered set is clear from the context.

\begin{remark} \label{rem:higher-power-wed} {\bf (1)} For a $C^2$ function $u$ the inequality \eqref{eq:chi_u-m-positive} is equivalent to the inequalities 
\[
	(\chi + dd^cu)^k \wed \alpha^{n-k} \geq 0 \quad \mbox{for } k=1,...,m.
\]
This fact can be seen as follows: for any real $(1,1)$-form $\tau \in \Gamma_m(\alpha)$ and $1\leq k\leq m$,
\[
	\frac{\tau^k \wed \alpha^{n-k}}{\alpha^n} = \left(\inf_\ga \left\{\frac{\tau\wed \ga^{k-1} \wed \alpha^{n-k}}{\alpha^{n}} \right\}\right)^{k},
\]
where $\ga$ is taken such that $\ga \in \Gamma_m(\alpha)$ and $\ga^k\wed\alpha^{n-k}/\alpha^n =1$. In other words, $u \in SH_{\chi,m}(\alpha, \Omega)$ if and only if $\chi + dd^c u\in \overline{\Gamma_m(\alpha)}$ at any given point in $\Omega$.

\noindent{\bf (2)} There are another definitions for $m-\alpha$-subharmonic functions. The first one is suggested by B\l ocki \cite{blocki05} and the second one is given by Lu \cite[Definition 2.3]{chinh13b} in a more general setting. All definitions are equivalent in the case of $m-$subharmonic functions, i.e. $\alpha = dd^c |z|^2$. Later on, by Lemma~\ref{def-2-omega-sh}, we will find that our definition is equivalent to the one in \cite{chinh13b}.
\end{remark}

We list here some basic properties of $(\chi,m)-\alpha$-subharmonic functions. 

\begin{prop} \label{prop:closure-max}
Let $\Omega$ be a bounded open set in $\bC^n$. 
\begin{enumerate}
\item[(a)]
If $u_1 \geq u_2 \geq  \cdots$ is a decreasing sequence of $(\chi,m)-\alpha$-subharmonic functions, then $u := \lim_{j\to \infty} u_j$ is either $(\chi,m)-\alpha$-subharmonic or $\equiv -\infty$.
\item[(b)] 
If $u, v$ belong to $SH_{\chi,m}(\alpha)$, then so does $\max\{u,v\}$.
\item[(c)]
Let $\{u_\alpha\}_{\alpha \in I} \subset SH_{\chi,m}(\alpha)$ be a family locally uniformly bounded above. Put $u(z) := \sup_\alpha u_\alpha(z)$. Then, the upper semicontinuous regularisation $u^*$ is $(\chi,m)-\alpha$-subharmonic.
\end{enumerate}
\end{prop}

\begin{proof} It is enought to verify $\tilde \alpha-$subharmonicity for every $\tilde\alpha^{n-1} = \ga_1\wed \cdots\wed \ga_{m-1}\wed \alpha$ with $\ga_i \in \Gamma_m(\alpha)$.  Once $\tilde \alpha$ is fixed the proof follows from Appendix (Proposition~\ref{proa:limit-stable}, Corollary~\ref{cora:limsup-closed}).
\end{proof}

\section{Potential estimates in a small ball}
\label{sec:2}

In this section we  develop potential theory for $(\chi,m)-\alpha$-subharmonic functions in a Euclidean ball, where $\alpha$ is conformal to a K\"ahler metric on this ball. To do this we fix a ball $B:= B(z,r)\subset\subset \Omega$ with the small radius, where $\Omega$ is a bounded open set in $\bC^n$. We also fix  a smooth function $G: \bar B \to \bR$ such that $\omega:= e^G \alpha$ is K\"ahler metric, i.e.,
\[ \label{eq:conformal-kahler}
	d (e^G \alpha) =0 \quad \mbox{on } \bar B.
\]
Notice that by Definition~\ref{def:ga-m-sub} we have $SH_{\chi,m}(\alpha) \equiv SH_{\chi,m}(\omega)$ as $\Gamma_m(\alpha) \equiv \Gamma_m(\omega)$. 

 First, we will work with an apparently smaller class of functions.

\begin{defn} Let $v$ be a $(\chi,m)-\alpha$-subharmonic function in a neighborhood of $\bar B$. $v$ is said to belong to $\cA$  if there exists sequence of smooth $(\chi,m)-\alpha$-subharmonic functions $v_j \in C^\infty(\bar B)$ decreasing point-wise to $v$ in $B$ as $j$ goes to $\infty$.
\end{defn}

For simplicity we also assume in this section that for every $z\in \bar \Omega$,
\[ \label{eq:positivity-chi1}
	\chi(z) \in \Gamma_m(\alpha),
\]
(otherwise we replace $\chi$ by $\tilde \chi := \chi + C dd^c \rho$ for a strictly plurisubharmonic  function $\rho$ in $\bar \Omega$ and $C>0$ large.) Since $\bar B$ is compact,  there exist $0< c_0 \leq 1$, depending on $\chi, \alpha, \bar B$, such that
$$	\chi - c_0 \alpha \in \Gamma_m(\alpha). 
$$
Throughout the paper we often write
$$ \chi_u := \chi + dd^c u \quad \mbox{for } u \in SH_{\chi, m}(\alpha). 
$$

\subsection{Hessian operators}
\label{sec:2.1}

According to the results in \cite{KN3} for any $v_1, ...,v_m \in \cA \cap C^0(\bar B)$, the wedge product $$\chi_{v_1} \wed \cdots \wed \chi_{v_m} \wed \alpha^{n-m}$$ is a well defined  positive Radon measure for a general Hermitian metric $\alpha$. However, to define the wedge product for $v_i \in \cA \cap L^\infty(\bar B)$ we will need the K\"ahler property of $\omega = e^G \alpha$ in \eqref{eq:conformal-kahler}.

Following ideas of Bedford-Taylor \cite{BT82}, by a simple modification, we can define the wedge product for $v_i \in \cA \cap L^\infty(\bar B)$ as follows. Fix a strictly plurisubharmonic  function $\vphi$  in a neighborhood of $\bar B$ such that
\[\notag
	\tau: = dd^c \vphi - \chi >0.
\]
Let us denote $w_i := v_i+ \vphi$. Then $w_i$ is $m-\omega$-subharmonic and bounded in $\bar B$, which is also in the class $\cA$. Since $\omega$ is K\"ahler, we define inductively  for $1\leq k \leq m$,
\[\label{eq:induction-process}
	dd^c w_k \wed \cdots \wed dd^cw_1 \wed \omega^{n-m} := dd^c (w_k dd^c w_{k-1} \wed \cdots \wed dd^cw_1 \wed \omega^{n-m}).
\]
The resulted wedge product is a positive $(n-m+k,n-m+k)-$current. Then,  one puts
\[ \label{eq:wed-prod1}
	dd^c w_k \wed \cdots \wed dd^cw_1 \wed \alpha^{n-m}:= e^{(m-n) G} dd^c w_k \wed \cdots \wed dd^cw_1 \wed \omega^{n-m}.
\]
We see that local properties that hold for a positive current on the right hand side will be preserved to the positive currents on the left hand side. Finally, using a formal expansion, we set 
\[ \label{eq:wed-prod-2}\begin{aligned}
&	\chi_{v_1} \wed \cdots\wed \chi_{v_m} \wed \alpha^{n-m} \\ 
&:= \sum_{\{i_1, ...,i_k\} \subset \{1,...,m\}} (-1)^{n-k} dd^c w_{i_1} \wed \cdots \wed dd^c w_{i_k} \wed \alpha^{n-m} \wed \tau^{n-k}.
\end{aligned}
\]
This is an honest equality in the case $v_i's$ are smooth functions. The right hand side still makes sense, when $v_i's$ are only bounded, by \eqref{eq:induction-process} and \eqref{eq:wed-prod1}. Thus, we get the wedge product on the left hand side is a well-defined $(n,n)-$positive current.

We also observe that the equation \eqref{eq:wed-prod-2} does not depend on the choice of a strictly plurisubharmonic  function $\vphi$ satisfying $dd^c \vphi -\chi >0$. Moreover,  let $T = \chi_{v_1} \wed \cdots \wed \chi_{v_k}\wed \alpha^{n-m}$ for $v_i \in \cA \cap L^\infty(\bar B)$ and $w\in \cA\cap L^\infty(\bar B)$. Then, we have
$$ (\chi + dd^c w) \wed T = \chi \wed T + dd^c w \wed T.$$
In other words, the definition of the wedge product obeys the linearity as in the smooth case.

\begin{remark} \label{non-def} If we do not assume $d \alpha = 0$ (or $d (e^G \alpha) =0$ for some function $ G$), then in the inductive definition we cannot get rid of the extra terms, e.g.,
\[ \notag
	dd^c v_k \wed \cdots \wed dd^cv_1 \wed dd^c \alpha^{n-m}.
\]
As $dd^c v_i$ is not $(1,1)-$positive current, we do not know how to define the wedge product for bounded functions $v_i$ in $\cA$ once the power of the base $\alpha$ is less than $n-m$. It is worth to mention that if  $v_i's$ are continuous and belong to $\cA$, then we can use the uniform convergence of potentials to define wedge product as in \cite{KN3}. 
\end{remark}

As in  \cite{KN3} the Chern-Levine-Nirenberg (CLN) inequalities are proved quickly in the present setting.

\begin{lem} \label{lem:CLN-ineq}Let $u_1, ...,u_m \in \cA \cap L^\infty(\bar B)$.  Let $K \subset \subset B$ be a compact set. Then,
\[\notag
	\int_K \chi_{u_1} \wed \cdots \wed \chi_{u_m} \wed \alpha^{n-m} 
	\leq C,
\]
where $C$ depends on $\alpha, K, B \|u_1\|_{L^\infty(B)}, ...,\|u_m\|_{L^\infty(B)}$.
\end{lem}

\begin{proof} Since $\omega=e^G\alpha$ is K\"ahler and $ G$ is bounded on $\bar B$,
$$\int_K \chi_{u_1} \wed \cdots \wed \chi_{u_m} \wed \alpha^{n-m} \leq e^{(n-m)\sup_{\bar B} |G|} \int_K \chi_{u_1} \wed \cdots \wed \chi_{u_m} \wed \omega^{n-m}.$$
Hence, the inequality follows from formulas \eqref{eq:wed-prod1}, \eqref{eq:wed-prod-2} and the classical argument by integration by parts (see \cite[Proposition~2.9]{KN3}).
\end{proof}

The following Bedford-Taylor convergence theorems are crucial in our approach.

\begin{thm}[]\label{thm:BT-convergence} Let $\{u_1^j\}_{j\geq 1},...,\{u_m^j\}_{j\geq 1} \subset \cA \cap L^\infty(\bar B)$ be decreasing (or increasing) sequences which converge point-wise to $u_1,...,u_m \in  \cA \cap L^\infty(\bar B)$, respectively. 
Then, the sequence of positive measures
$$(\chi + dd^c u_1^j) \wed \cdots \wed (\chi + dd^cu_m^j) \wed \alpha^{n-m} $$ converges weakly to the positive measure $$ (\chi + dd^c u_1) \wed \cdots \wed (\chi + dd^cu_m) \wed \alpha^{n-m} $$
as $j \to \infty$.
\end{thm}

\begin{proof}Recall that $\omega:=e^G \alpha$ is a K\"ahler form on $\bar B$. By definitions \eqref{eq:wed-prod1} and \eqref{eq:wed-prod-2} it is enough to show that if decreasing sequences of bounded $m-\omega$-suharmonic functions $\{v_1^j\}_{j\geq 1}, ...,\{v_m^j\}_{j\geq 1}$ converge to bounded $m-\omega$-subharmonic functions $v_1,...,v_m$, respectively, then the sequence of $(n,n)-$positive currents 
$
	dd^c v_1^j \wed \cdots \wed dd^c v_m^j \wed \omega^{n-m}
$
weakly converges to $dd^c v_1 \wed \cdots \wed dd^c v_m\wed \omega^{n-m}$. Therefore, the theorem follows by an easy adaption of arguments of Bedford-Taylor \cite{BT82}.
\end{proof}

Let us define the notion capacity associated with Hessian operators which plays important role in the study of bounded $(\chi,m)-\alpha$-subharmonic functions. For a Borel set $E\subset B$,
\[ \label{eq:cap}
	cap(E):= \sup\left\{\int_E (\chi + dd^c v)^m\wed \alpha^{n-m} : v \in \cA,   0\leq v \leq 1\right\}.
\]
We first observe that this capacity is equivalent to another capacity.
\begin{lem} \label{lem:cap-equiv}
For a Borel set $E\subset B$, 
\[ \label{eq:cap_m}
	{\bf c}_m(E) := \sup\left\{ \int_E (dd^c w)^m \wedge \alpha^{n-m}: w\in \cA_0, 0\leq w \leq 1\right\},
\]
where $\cA_0$ is the class $\cA$ with $\chi \equiv 0$. Then, there exists a constant $C$ depending on $\chi, \alpha$ such that
\[\notag
	\frac{1}{C} cap(E) \leq {\bf c}_m(E) \leq C cap(E)
\]
for any Borel set $E\subset B$.
\end{lem}

\begin{proof} Since $\chi \leq dd^c \vphi$ for some smooth plurisubharmonic  function on $\bar B$, the first inequality follows. To show the second one, we need to use the positivity of $\alpha$. By \eqref{eq:positivity-chi1} there is  a  constant $C>0$ such that
\[\notag
	\chi - \frac{1}{C} dd^c \rho \in \Gamma_m(\alpha),
\]
where $\rho = |z|^2 -r^2\leq 0$. We can choose $C$ such that $|\rho/C| \leq 1/2$. Take a function $0 \leq w \leq 1/2$  in $\cA_0$, then it is easy to see that
\[\notag
	\int_E (dd^cw)^m \wedge \alpha^{n-m}  
	\leq \int_E \left(\chi + dd^c(w - \frac{\rho}{C})\right)^m \wedge \alpha^{n-m} \leq cap(E).
\]
Hence, ${\bf c}_m(E) \leq 2^m  cap(E)$. 
\end{proof}

\begin{cor} \label{cor:quasi-cont-a}
Let $u \in \cA \cap L^\infty(\bar B)$. Then, $u$ is quasi-continuous with respect to the capacity $cap(\cdot)$.
\end{cor}

\begin{proof} Observe that $v:= u + \vphi$ is $m-\alpha$ subharmonic for some smooth plurisubharmonic  function $\vphi$ on $\bar B$. Therefore, $v$ is also approximated by a decreasing sequence of smooth $m-\alpha$ subharmonic functions. By the arguments in Bedford-Taylor  \cite{BT82} adapted  to the case $\omega = e^G\alpha$ (see similar arguments in Lemma~\ref{lema:quasi-con}), we get that $v$ is quasi-continuous with respect to ${\bf c}_m(\cdot)$. By Lemma~\ref{lem:cap-equiv} the proof is completed.
\end{proof}

The next consequence is an inequality between volume and capacity. 

\begin{lem} \label{lem:vol-cap}
 Fix $1<\tau < n/(n-m)$. There exists a constant $C(\tau)$ such that for any Borel set $E\subset B$
\[ \label{eq:vol-cap-ineq}
	V_\alpha(E) \leq C(\tau) \big[cap (E)\big]^{\tau},
\]
where $V_\alpha(E):=\int_E \alpha^n.$
\end{lem}

The exponent here is optimal because if we take $\alpha = dd^c |z|^2$, then the explicit formula for ${\bf c}_m(B(0,s))$ in $B = B(0, r)$ with $0< s<r$, provides an example.

\begin{proof} From \cite[Proposition~2.1]{dk14} we knew that $V_\alpha(E) \leq C [{\bf c}_m(E)]^{\tau}$  with ${\bf c}_m(E)$ defined in \eqref{eq:cap_m}. Note that the argument in \cite{dk14} remains valid for  non-K\"ahler $\alpha$  since the mixed form type inequality used there still holds by stability estimates for the Monge-Amp\`ere equation. Thanks to Lemma~\ref{lem:cap-equiv} the proof follows.
\end{proof}

\subsection{Comparison principles in $\cA \cap L^\infty(\bar B)$}
\label{sec:2.2}

For simplicity if $u,v \in \cA\cap L^\infty(\bar B)$ we write
\[
	u \geq v \quad\mbox{on } \d B\quad \mbox{meaning that}\quad
	\liminf_{z \to \d B} (u - v) \geq 0.
\]

\begin{lem} \label{lem:wcp-a} Let $u, v \in SH_{\chi,m}(\alpha) \cap L^\infty(\bar B)$ be such that $u\geq v$ on $\d B$. Let $T = \chi_{v_1} \wedge \cdots \wedge \chi_{v_{m-1}} \wedge \alpha^{n-m}$ with $v_i \in SH_{\chi,m}\cap L^\infty(\bar B)$. Then,
\[ \notag
	\int_{\{u<v\}} dd^c v \wedge T \leq \int_{\{u<v\}}  dd^c u \wedge T + 
	\int_{\{u<v\}} (v-u) dd^c T.
\]
\end{lem}

Notice that by the equations \eqref{eq:wed-prod1} and \eqref{eq:wed-prod-2} 
\begin{align*} dd^c T &=  dd^c \left( e^{(m-n) G} \chi_{v_1} \wedge \cdots \wedge \chi_{v_{m-1}} \wedge \omega^{n-m}\right) \\
&= dd^c \left( e^{(m-n) G}\chi_{v_1} \wedge \cdots \wedge \chi_{v_{m-1}}\right) \wedge \omega^{n-m}.
\end{align*}
where $\omega = e^G \alpha$ is a fixed K\"ahler form in \eqref{eq:conformal-kahler}.

\begin{proof} By replacing $u$ by $u+\delta$ for $\delta>0$ and then letting $\delta\searrow 0$ we will work with $\{u< v\}\subset\subset  K \subset \subset B$, where $K$ is an open set. By the CLN inequality (Lemma~\ref{lem:CLN-ineq})
\[ \notag
	\int_{K}\|dd^cT\| <+\infty.
\]
By Theorem~\ref{thm:BT-convergence}, Corollary~\ref{cor:quasi-cont-a}, and arguments in \cite{BT87}  we get that
\[ \label{wcpa-eq2}
	{\bf 1}_{\{u<v\}} dd^c \max\{u,v\} \wedge T = {\bf 1}_{\{u<v\}} dd^c v\wedge T
\]
as two measures. Since $\{u+\vepsilon < v\} \subset \subset K$ for $\vepsilon>0$, Stokes' theorem gives
\[ \notag
\begin{aligned}
&	\int_{K} dd^c \max\{u+\vepsilon,v\} \wedge T \\
&= \int_{\d K} d^c u \wedge T + \int_K d^c \max\{u+\vepsilon,v\} \wedge d T \\
&= \int_{\d K} d^c u \wedge T + \int_{\d K}  u \wedge dT + \int_{K} \max\{u+\vepsilon, v\} dd^c T \\
&= \int_K dd^c u \wedge T - \int_{K} u dd^c T + \int_{K} \max\{u+\vepsilon, v\} dd^c T \\
&=	\int_K dd^c u \wedge T + \int_{\{u+\vepsilon < v\} \cap K} ( v - u) dd^c T + 
\vepsilon \int_{\{u+\vepsilon \geq v\} \cap K} dd^c T.
\end{aligned}
\]
Moreover, by the identity \eqref{wcpa-eq2}
\[  \notag
\begin{aligned}
&	\int_{\{u+\vepsilon <v\}} dd^c v \wedge T \\&= 
	\int_{\{u+\vepsilon <v\}} dd^c\max\{ u+ \vepsilon,v\} \wedge T \\
&= \int_K dd^c\max\{ u+ \vepsilon,v\} \wedge T - \int_{\{u+\vepsilon \geq v\} \cap K} dd^c\max\{ u+ \vepsilon,v\} \wedge T \\
&\leq \int_K dd^c\max\{ u+ \vepsilon,v\} \wedge T - \int_{\{u+\vepsilon > v\} \cap K} dd^c u \wedge T .
\end{aligned}
\]
Thus, it follows that
\[ \notag\begin{aligned}
	\int_{\{u+\vepsilon <v\}} dd^c v \wedge T
&\leq 	\int_{\{u+ \vepsilon \leq v\}} dd^c u \wedge T +  \int_{\{u+\vepsilon < v\}} ( v - u) dd^c T \\
&\quad + 
\vepsilon \int_{K} \|dd^c T\|.
\end{aligned}
\]
Letting $\vepsilon \searrow 0$ we get the desired inequality.
\end{proof}

In the Hermitian setting due to the torsion of $\alpha$ and $\chi$, the classical comparison principle no longer holds. However, its weak versions in \cite{DK12} and \cite{KN1} are enough for several applications. We state the local counterparts of those. 

Let $D_1, D_2$ be two constants such that on $\bar B$,
\[ \label{eq:alpha-chi-bound}
\begin{aligned}
	-D_1 \alpha^2 \leq dd^c \alpha \leq D_1 \alpha^2, \quad
	-D_1 \alpha^3 \leq d \alpha\wed d^c \alpha \leq D_1 \alpha^3; \\ 
	-D_2 \alpha^2 \leq dd^c \chi \leq D_2 \alpha^2, \quad
	-D_2 \alpha^3 \leq d \chi \wed d^c \chi \leq D_2 \alpha^3. 
\end{aligned}
\]

\begin{lem}\label{lem:wcp-b} Let $u,v \in \cA \cap L^\infty(\bar B)$ be such that $u \geq v$ on $\d B$.
Assume that 
$
	d = \sup_{\bar B} (v - u) >0.
$
and  $D_1D_2 \sup_{\{u<v\}}(v-u) \leq 1$. Then,
 \[ \notag \begin{aligned}
&	\int_{\{u<v\}} (\chi + dd^c v)^m \wed \alpha^{n-m}  
\leq \int_{\{u<v\}} (\chi + dd^c u)^m\wed \alpha^{n-m} + \\
&+ C D_1 D_2 \sup_{\{u<v\}}(v-u) \sum_{k=0}^{m-1} \int_{\{u<v\}} (\chi + dd^c u)^k  \wed \alpha^{n-k}.
\end{aligned} 
\]
The constant $C$ depends only on $n,m$.
\end{lem}

\begin{proof} We used repeatedly Lemma~\ref{lem:wcp-a} (for $T= \chi_u^k \wed \chi_v^{l} \wed \alpha^{n-k-l}$, $k+l \leq m-1$), and bounds in \eqref{eq:alpha-chi-bound} to replace $v$ by $u$.  Thanks to  results in \cite[Section 2]{KN3} the arguments go through for general Hessian operators with respect to the Hermitian metric $\alpha$.
\end{proof}

Recall from \eqref{eq:positivity-chi1} that  there exists $0< c_0 \leq 1$, depending on $\chi, \alpha, \bar B$, such that 
\[ \label{eq:positivity-chi2}
	\chi - c_0 \alpha \in \Gamma_m(\alpha). 
\]
The weak comparison principle is a crucial tool in pluripotential theory approach to study weak solutions of Hessian type equations \cite{KN1,KN2,KN3}. We state a local version.

\begin{lem} \label{lem:wcp-c} Let $u,v \in \cA \cap L^\infty(\bar B)$ be such that $u \geq v$ on $\d B$. 
Assume that 
$
	d = \sup_{B} (v - u) >0.
$
Fix $0< \vepsilon < \min\{1/2, d/(1+ 2\|v\|_{L^\infty)}\}$. Denote  
$
S(\vepsilon) = \inf_B [u - (1-\vepsilon) v]$, and  for $s>0$, \[U(\vepsilon,s):= \{u<(1-\vepsilon) v + S(\vepsilon) +s\}. \notag \]
Then, for $0<s<  (c_0\vepsilon)^3/(16D_1D_2)$,
\[ \notag
	\int_{U(\vepsilon,s)} \big(\chi +(1-\vepsilon)dd^cv\big)^m \wedge \alpha^{n-m} \leq \left(1+ \frac{C s}{(c_0\vepsilon)^m}\right) \int_{U(\vepsilon,s)} (\chi + dd^cu)^m \wedge \alpha^{n-m} .
\]
The constant depends on $n, m, D_1, D_2$.
\end{lem}

\begin{proof} We only give here a brief argument as it is very similar to the one of \cite[Theorem 2.3]{KN1}. Set for $0\leq k\leq m$,
\[ \notag
	a_k:= \int_{U(\vepsilon(s))} \chi_u^k \wed \alpha^{n-k}.
\]
Then, 
\[ \notag
	(c_0 \vepsilon) a_k \leq \vepsilon  \int_{U(\vepsilon,s)} \chi_u^k \wed \chi \wed \alpha^{n-k-1} \leq \int_{U(\vepsilon,s)} \chi_u^k \wed \chi_{(1-\vepsilon)v} \wed \alpha^{n-k-1}.
\]
By Lemma~\ref{lem:wcp-a}
\[ \notag
\int_{U(\vepsilon,s)} \chi_u^k \wed \chi_{(1-\vepsilon)v} \wed \alpha^{n-k-1} 
\leq \int_{U(\vepsilon,s)} \chi_u^{k+1} \wed \alpha^{n-k-1} + R,
\]
where 
$
	R = \int_{U(\vepsilon,s)} [(1-\vepsilon) v + S(\vepsilon) + s - u]
	dd^c \left( \chi_u^k \wed \alpha^{n-k-1}\right)$. It is bounded by 
\[ \notag
	R	 \leq s D_1D_2 (a_k + a_{k-1} + a_{k-2}),
\]
where we simply understand $a_k \equiv 0$ for $k<0$. To be honest, here we used \cite[Lemma 2.3]{KN3}, hence we should multiply the right hand side with a constant $C_{m,n}>0$ depending only on $m,n$. This is no harm as we could adjust the definitions of $D_1, D_2$. 

Thus, for $0<s< \delta:=\frac{(c_0 \vepsilon)^3}{D_1D_2}$,
$	(c_0\vepsilon) a_k \leq \delta (D_1D_2) (a_k+a_{k-1}+ a_{k-2})+a_{k+1}.
$
The rest goes in the same way as in \cite[Theorem 2.3]{KN1}.
\end{proof}

The following result is obvious if potential functions are smooth.

\begin{cor} \label{cor:do-prin} Let $u, v \in \cA\cap L^\infty(\bar B)$ be such that $u \geq v$ on $\d B$. Suppose that $\chi_u^m \wedge \alpha^{n-m} \leq \chi_v^m\wed \alpha^{n-m}$ in $B$. Then, $u\geq v$ on $\bar B$.
\end{cor}

\begin{proof} It follows from the proof of \cite[Corollary 3.4.]{KN1} with obvious modifications. The reason is that there exists a $C^2$ strictly plurisubharmonic  function on $\bar B$.
\end{proof}

We have proved the comparison principle (Lemma~\ref{lem:wcp-c}) and volume-capacity inequality (Lemma~\ref{lem:vol-cap}). The following uniform a priori estimate is proved in the identically way as \cite[Theorem 3.10]{KN3}. 

\begin{thm} \label{thm:a-priori-estimate} 
Let $u, v \in \cA \cap L^\infty(B)$ be such that
$$ \liminf_{z\to\d B} (u - v) \geq 0, \quad
d:=\sup_B (v- u) >0.$$ 
Let us fix the following constants: 
$$	 \begin{aligned}
&	p>n/m, \quad	0< \tau < \frac{p -\frac{n}{m}}{p(n-m)},  
\quad \tau^*= \frac{(1+m \tau)p}{p-1};\\
&	0<\vepsilon< \min\{1/2, d/3(1+\|v\|_\infty)\}; \\
&	\vepsilon_0 := \frac{1}{3} \min\{(c_0\vepsilon)^m, \frac{(c_0\vepsilon)^3}{16D_1D_2}\}.
\end{aligned} $$
Suppose that 
$(\chi +dd^cu)^m \wedge \alpha^{n-m} = f \alpha^n$ on $B$ with $f \in L^p(B,\alpha^n)$.  Assume that $v$ is continuous and put 
$$	U(\vepsilon, s) = \{u<(1-\vepsilon) v + \inf_B [u -(1-\vepsilon) v] +s \}.$$ 
Then, there exists a constant $C = C(\tau, \alpha, B)$ such that for every $0<s<\vepsilon_0$, 
$$	s\leq C (1+ \|v\|_{L^\infty (B)}) \|f\|_{L^p(B)}^\frac{1}{m} 
	\left[V_\alpha(U(\vepsilon,s)) \right]^\frac{\tau}{\tau^*},$$
where $V_\alpha(E) = \int_E \alpha^n$ for a Borel set $E$.
\end{thm}

Notice that from assumptions, the sub-level sets near the infimum point will be non-empty and relatively compact in the ball $B$. The restriction on the class $\cA$ will be relaxed later (see Remark~\ref{rmk:approximation-after}).

\subsection{The Dirchlet problems on $\bar B$}

Consider the Dirichlet problem with the right hand side in $L^p(B)$, $p>n/m$. Notice that $n/m$ is the optimal exponent.

\[ \label{eq:wdp-ball-cont} \begin{aligned} 
&u \in \cA \cap C^0(\bar B), \\
&(\chi + dd^c u)^m \wedge \alpha^{n-m} = f \alpha^{n}, \\
&u = \vphi \in  C^0(\d B).
\end{aligned}
\]

\begin{lem} \label{lem:stability1} Let $f,g$ be non-negative functions in $L^p(B), p>n/m$. Let $\vphi, \psi \in C^0(\d B)$. Suppose that $u,v$ are solutions to the corresponding Dirichlet problem \eqref{eq:wdp-ball-cont} with the datum $(f, \vphi)$ and $(g, \psi)$. Then, 
$$	\|u-v\|_{L^\infty (B)} \leq \sup_{\d B} |\vphi -\psi | + C \|f-g\|_{L^p (B)}^\frac{1}{m},
$$
where $C$ depends only on  $p$ and the diameter of $B$.
\end{lem}

\begin{proof} We use an idea in \cite{dk14}, which used the uniform {\em a priori} estimate for Monge-Amp\`ere  equation due to Ko\l odziej \cite{kol96}. The proof here is similar to \cite[Theorem 3.11]{cuong-thesis}.  Put $h= |f-g|^\frac{n}{m}$ in $B$. 
It follows that $h \in L^\frac{pm}{n}(B)$, where $\frac{pm}{n}>1$. Moreover,
$$	\|h\|_{L^\frac{pm}{n} (B)}^\frac{1}{n} = \|f-g\|_{L^p(B)}^\frac{1}{m}. 
$$
By a theorem in \cite{kol96}, there exists 
$\rho \in PSH(B) \cap C^0(\bar B)$ solving 
$$	(dd^c \rho)^n  = h \alpha^n, \quad 
	\rho_{|_{\d B}} = 0.
$$
We also have
$$
	\| \rho \|_{L^\infty} \leq C  \|h\|_{L^\frac{pm}{n}(B)}^\frac{1}{n}
	= C \|f -g\|_{L^p(B)}^\frac{1}{m},
$$
where $C = C(m,n, p,B, \alpha)$ a uniform constant.  Furthermore, by the mixed-form  inequality, 
$$	(dd^c \rho)^m \wedge \alpha^{n-m} \geq h^\frac{m}{n} \alpha^n
	= |f -g| \alpha^n.
$$
Therefore,
\begin{align*}
	[\chi_u+ dd^c\rho]^m\wedge \alpha^{n-m} 
&	\geq \chi_u^m\wedge \alpha^{n-m} + (dd^c\rho)^m \wedge \alpha^{n-m} \\
&	\geq f \alpha^n + |f-g| \alpha^n \\
&	\geq g \alpha^n.
\end{align*}
Since $\rho \leq 0$ in $\bar B$, it follows from the domination principle (Corollary~\ref{cor:do-prin}) that
$	u+\rho \leq v + \sup_{\d B} |u-v|. 
$
Hence, 
$$	u- v \leq -\rho + \sup_{\d B} |u-v| 
	\leq \sup_{\d B} |u-v| 
	+ C \|f-g\|_{L^p(B)}^\frac{1}{m}.
$$
Similarly,
$	v - u \leq \sup_{\d B} |u-v| + C \|f-g\|_{L^p(B)}^\frac{1}{m}.
$
Thus, the theorem follows.
\end{proof}

We also need another stability estimate for solutions whose Hessian operators are in $L^p$, $p>n/m$. 

\begin{lem} \label{lem:stability2} Under the assumptions of Lemma~\ref{lem:stability1} there exist a uniform constant $C = C(p,m,n, \|f\|_p, \|g\|_p)$ and a constant $a = a(p,m,n)> 0$ such that
\[ \notag
	\|u- v\|_{L^\infty(B)} \leq \sup_{\d B} |\varphi - \psi| + C \|u-v\|_{L^1(B)}^a.
\]
\end{lem}

\begin{proof} Having Theorem~\ref{thm:a-priori-estimate} we can repeat the proof of \cite[Theorem~3.11]{KN3} two times, one for the pair $u+\sup_{\d B} |\vphi-\psi|$ and $v$ and another for the pair $v +\sup_{\d B} |\vphi-\psi|$ and $u$. 
\end{proof}

We get from the existence of smooth solutions (Theorem~\ref{thm:intr-1}) and stability estimates (Lemma~\ref{lem:stability1}) existence of weak solutions.

\begin{thm} Let $0\leq f\in L^p(B)$ with $p>n/m$. Then, there exists a unique solution to the Dirichlet problem \eqref{eq:wdp-ball-cont}.
\end{thm}



The last ingredient to prove the approximation property for $(\chi,m)-\alpha$-subharmonic functions is the existence of smooth solutions for a Hessian type equation.

\begin{lem} \label{lem:smooth-heq-type} Let H be a smooth function on $\bar B$ and $\varphi \in C^\infty(\d B)$. Then, there exists a unique $u \in SH_{\chi,m}(\alpha) \cap C^\infty(\bar B)$ solving the Hessian equation 
\[ \notag\begin{aligned}
	(\chi + dd^c u)^m \wedge \alpha^{n-m} = e^{u + H} \alpha^n, \\
	u = \varphi \quad \mbox{on } \d B.
\end{aligned}
\]
\end{lem}

\begin{proof} The right hand side depends also on $u$ but with the right sign. We solve the equation by the continuity method as in the proof of Theorem~\ref{thm:intro-3}, provided second order apriori estimates. The $C^0-$estimate easily follows by considering the maximum point and the minimum point of the solution. So does the $C^1-$estimate on the boundary. The proof of $C^1-$estimate at an interior point will be affected at equations \eqref{grad-eq-t1a} and \eqref{grad-eq-t1b} in Section~\ref{sec:c1}. The extra terms appear in these equations are $O(|\nabla u|^2)$. So this will not affect the conclusion of the inequality \eqref{grad-eq-t12}. Therefore, we will get $C^1-$estimate. The $C^2-$estimate at an interior point goes through as in Section~\ref{sec:c2}, as it is explained in \cite[Lemma 3.18]{KN3}). For the other $C^2-$estimates at a boundary point, the equation \eqref{eq:diff-c2-extra} contains a bounded term $O(|\nabla u|)$ by the $C^1-$estimate. Therefore, the equality \eqref{use-diff-eq} will still hold and we get the desired estimates.
\end{proof}

\begin{lem} \label{lem:stability-heq-type}Let $0 \leq f \in L^p(B)$, $p>n/m$, and $\varphi \in C^0(\d B)$. Let $\{f_j\}_{j\geq 1}$ be smooth and positive functions on $\bar B$, converging in $L^p(\bar B)$ to $f$ as $j \to +\infty$. Let $\varphi_j \in C^{\infty}(\d B)$ converge uniformly to $\varphi$. Assume that 
\[ \notag\begin{aligned}
	\chi_{u_j}^m \wedge \alpha^{n-m} = e^{u_j} f_j \alpha^{n}, \\
	u_j	= \varphi_j \quad \mbox{on } \d B.
\end{aligned}
\]
Then, $u_j$ converges uniformly to $u \in \cA \cap C^0(\bar B)$, which is the unique solution in $\cA \cap C^0(\bar B)$ of 
$$	\begin{aligned}
	\chi_u^m \wedge \alpha^{n-m} = e^u f \alpha^n, \\
	u = \varphi \quad \mbox{on } \d B.
\end{aligned}$$
\end{lem}

\begin{proof} Observe that $u_j$ is uniformly bounded above.  It follows that the right hand side of equations are uniformly bounded in $L^p$. Applying Lemma~\ref{lem:stability1} for $\psi=0$ and $g=0$, this gives the uniform bound for $u_j$. Then, by compactness of the sequence $u_j $ in $L^1$ and Lemma~\ref{lem:stability2} we get a continuous solution by passing to the limit. The uniqueness follows as in  \cite[Lemma 2.3]{cuong15}.
\end{proof}

\subsection{Approximation property on $\bar B$}
\label{sec:approximation-property}

We have all ingredients which are needed to prove the main theorem of this section. By using results of Pli\'s \cite{plis}, Harvey - Lawson - Pli\'s \cite[Theorem 6.1]{HLP} also proved this theorem in the case $\chi \equiv 0$ and $\alpha$ being K\"ahler.

\begin{thm} \label{thm:approximation} Let $u$ be $(\chi,m)-\alpha$-subharmonic in a neighborhood of $\bar B$. Then,
there exists a sequence of smooth functions $u_j \in SH_{\chi,m}(\alpha) \cap C^\infty(\bar B)$ such that $u_j$ decreases to $u$ point-wise in $B$ as $j$ goes to $+\infty$. 
\end{thm}

\begin{proof} We follow closely the proof of \cite[Lemma 3.20]{KN3}, which in turns uses the scheme introduced by Berman \cite{berman} and Eyssidieux-Guedj-Zeriahi \cite{egz13} (see also Lu-Nguyen \cite{chinh-dong}). 

By positivity assumption on $\chi \in \Gamma_m(\alpha)$ for every $z\in \bar B$ we have that $j \in SH_{\chi,m}(\alpha)$ for any constant $j$. As $\max\{u,-j\}$ belongs to $SH_{\chi,m}(\alpha)$, we may assume that $u$ is bounded. Since $u$ is upper semicontinuous on $\bar B$, there exists a sequence of smooth functions $\phi_j$ decreasing to $u$ on $\bar B$. Fix such an $h:= \phi_j$. Consider the envelope
\[ \label{env-eq1}
	\tilde h:= \sup\{v \in SH_{\chi,m}( \alpha) \cap L^\infty(B): v \leq h\}.
\]
Then, $\tilde h \in SH_{\chi,m} (\alpha)$ and $u \leq \tilde h \leq h$. Therefore, if $\tilde h \in \cA$, i.e. it has the approximation property, then so does $u$ by letting $h = \phi_j \searrow u$. We shall prove that the function $\tilde h$ can be approximated uniformly, and then the lemma will follow.

Since $h\in C^\infty(\bar B)$, we can write $\chi_h^m\wed\alpha^{n-m} = F\alpha^n$ with $F$ being a smooth function on $\bar B$. Let us denote $F_*=\max\{F,0\}$. We choose a sequence of smoothly non-negative functions $F_j$ decreasing uniformly to $F_*$ as $j\to \infty$. Fix such a $\tilde F := F_j \geq F_*$. By Lemma~\ref{lem:smooth-heq-type} we solve for $0<\vepsilon \leq 1$,
\[ \notag\begin{aligned}
\chi_{\tilde w_\vepsilon}^m \wed \alpha^{n-m} = e^{\frac{1}{\vepsilon}(\tilde w_\vepsilon -h)} [\tilde F + \vepsilon] \alpha^n, \\
\tilde w_\vepsilon = h \quad \mbox{on } \d B.
\end{aligned}
\]
By maximum principle, $\tilde w_\vepsilon \leq h$ and $\tilde w_\vepsilon$ is increasing as  $\vepsilon$ decreases to $0$. Keep $\vepsilon$ fixed, and take limit on both sides for $\tilde F = F_j  \to F_*$, i.e. letting $j\to\infty$, we get from Lemma~\ref{lem:stability-heq-type}, 
\[ \notag\begin{aligned}
	\chi_{w_\vepsilon}^m\wed \alpha^{n-m} = e^{\frac{1}{\vepsilon}(w_\vepsilon -h)} [F_* + \vepsilon] \alpha^n, \\
 w_\vepsilon = h \quad \mbox{on } \d B.
\end{aligned}
\]
Here $\tilde w_\vepsilon$ uniformly increases to $w_\vepsilon$. Thus, $w_\vepsilon \in \cA \cap C^0(\bar B)$ and $w_\vepsilon$ is increasing as $\vepsilon$ decreases to $0$. Since $w_\vepsilon \leq h$, the right hand side is uniformly bounded in $L^\infty(\bar  B)$. The monotone  sequence $w_\vepsilon$, bounded above by $h$, is a Cauchy sequence in $L^1( B)$. By Lemma~\ref{lem:stability2}, this sequence is also Cauchy in the uniform norm in $\bar B$. So, $w_\vepsilon$ uniformly increases to $w$ which satisfies
\[ \notag\begin{aligned}
	\chi_w^m\wed \alpha^{n-m} \leq {\bf 1}_{\{w=h\}} F_*\alpha^n,\\
	w = h \quad\mbox{on } \d B.
\end{aligned}
\]
In particular, $w \in \cA \cap C^0(\bar B)$. Now, we claim that $w=\tilde h$. The inequality $w \leq \tilde h$ is clear. One needs to verify that $w\geq \tilde h$ on $\{w<h\}$. Take a candidate $v$ in the envelope \eqref{env-eq1}, i.e, $v\leq h$. Observe that $\chi_w^m \wed \alpha^{n-m} = 0$ on $\{w<v\} \subset \{w<h\}$. By Corollary~\ref{cor:do-prin} it follows that $w$ is maximal on $\{w<h\}$. Thus, the set $\{w<v\}$ is empty, i.e., $w\geq v$. Since $v$ is arbitrary, so $w\geq \tilde h$. The claim follows and so does the theorem.
\end{proof}

\begin{remark} \label{rmk:approximation-after}
{\bf (a)} In  the proof we only used  the wedge product for continuous potential, so Theorem~\ref{thm:approximation} holds for a general Hermitian metric $\alpha$. In this case one should use a counterpart of \cite[Theorem 2.16]{KN3} instead of Corollary~\ref{cor:do-prin} in the last argument.
 
\noindent{\bf (b)} An immediate consequence is that the class $\cA$ coincides with  $SH_{\chi,m}(\alpha)$. 
\end{remark}

Thanks to the quasi-continuity and approximation property of $(\chi,m)-\alpha$- subharnonic functions we get an inequality similar to the one for plurisubharmonic  functions in Cegrell-Ko\l oldziej \cite{ck94}. 

\begin{prop}\label{prop:ck-ineq} Let $u, v\in SH_{\chi,m}(\alpha) \cap L^\infty(B)$. Let $\mu$ be a positive measure such that $\chi_u^m\wed \alpha^{n-m} \geq \mu$ and $\chi_v^m \wed \alpha^{n-m} \geq \mu$. Then
\[\notag
	\left(\chi + dd^c \max\{u,v\}\right)^m \wed \alpha^{n-m} \geq \mu.
\]
\end{prop}

\begin{proof} It is readily adaptable from  \cite[Theorem~1]{ck94} with an obvious change of notations.
\end{proof}

\section{The Dirichlet problem} \label{sec:3} 

On the complex manifold $M = \bar M \setminus \d M$ we define the class $SH_{\chi,m}(\alpha, M)$ in local coordinates. One main difference is that for an arbitrary real $(1,1)$-form $\chi$ on $M$, there are plenty of local $(\chi,m)-\alpha$-subharmonic functions on each local chart. However, the global class $SH_{\chi,m}(\alpha, M)$ may be empty, e.g. for negative $\chi$. Thus, the existence of a subsolution will guarantee that $	SH_{\chi,m}(\alpha) \mbox{ is non empty}$.

In this section we shall study weak solutions to the Dirichlet problem for the complex Hessian type equation.  As we pointed out in Section~\ref{sec:2.1} the assumption $\alpha$ is locally conformal K\"ahler metric on $M$ is needed to develop potential theory for bounded functions.

Fix the continuous right hand side density $0\leq f \in C^0(\bar M)$ and a continuous boundary data $\vphi \in C^0(\d M)$. Let us denote
$$ \mu:= f\alpha^n. $$
We wish to solve the Dirichlet problem:
\[ \label{dp-1}\begin{aligned}
&w \in SH_{\chi,m}(\alpha) \cap C^0(\bar M), \\
&(\chi + dd^c w)^m \wed \alpha^{n-m} = \mu, \\
&w = \varphi \quad \mbox{on } \d M.
\end{aligned}
\]
The $C^2$  subsolution $\rho$ to the equation \eqref{dp-1} satisfies:
$$	\chi_\rho:= \chi + dd^c \rho \in \Gamma_m(\alpha),$$
and
\[ \label{eq:sub-ass}
	(\chi + dd^c \rho)^m \wed \alpha^{n-m} \geq \mu, \quad
	\rho = \varphi \quad \mbox{on } \d M.
\]
By replacing $\chi$ by $\chi_\rho$ and $u$ by $u-\rho$ we can reduce the problem to the case of zero boundary data and $\chi \in \Gamma_m(\alpha)$ as follows:  
\[ \label{eq:dp-11}\begin{aligned}
&w \in SH_{\chi,m}(\alpha) \cap C^0(\bar M), \\
&(\chi + dd^c w)^m \wed \alpha^{n-m} = \mu, \\
&w = 0 \quad \mbox{on } \d M.
\end{aligned}
\]
Then $0$ is the subsolution to the equation~\eqref{eq:dp-11}, and  there exists $0< c_0 \leq 1$ such that
$$	\chi - c_0 \alpha \in \Gamma_m(\alpha). $$

\subsection{Envelope of continuous subsolutions} \label{sec:con-env}

By assumption \eqref{eq:sub-ass} the set 
\[ \notag
	\cS = \{v \in SH_{\chi,m}(\alpha) \cap C^0(\bar M): \chi_v^m\wed \alpha^{n-m} \geq \mu,\; v_{|_{\d M}} \leq 0 \}
\]
is not empty.  Hence, we define the 
envelope
\[ \label{eq:enve-1}
	u_0(z) := \sup_{v \in \cS} v(z).
\]
One expects that it will be a solution to the continuous Dirichlet problem. 

\begin{thm}  \label{thm:cont-sol}
If  $u_0$ is continuous, then it solves the Dirichlet problem \eqref{dp-1}.
\end{thm}

\begin{proof} We first have $u_0 \in S$ by Proposition~\ref{prop:closure-max}-(b) and  Proposition~\ref{prop:ck-ineq}. In particular,
$$	(\chi + dd^cu_0)^m \wed \alpha^{n-m} \geq  \mu.$$
It remains to show that $\chi_{u_0}^m \wedge \alpha^{n-m} = \mu$. Fix a small ball $B \subset M$ and find $w \in SH_{\chi,m}(\alpha) \cap C^0(\bar B)$ solving $w = u_0$ on $\d B$ and 
$$	(\chi + dd^c w)^m \wed \alpha^{n-m} = \mu \quad\mbox{in } B.$$
Hence, $w \geq u_0$ in $\bar B$. Consider
the lift $\tilde u \in \cS$ of $u_0$ with respect to this ball defined by 
$$ \tilde u = \begin{cases} \max{\{w, u_0\}} \quad \mbox{on } B, \\
	u_0 \quad \mbox{on } \bar M \setminus B.
\end{cases} $$
Thus, we have $\tilde u \in S$ and $u_0 \leq \tilde u$ in $B$. On the other hand by the definition of $u_0$ we have $\tilde u \leq u_0$. Thus, $u_0 = \tilde u$ in $B$, which means  $\chi_{u_0}^m \wedge \alpha^{n-m} =\mu$. This holds for any ball, so the theorem follows.
\end{proof}

\begin{remark} For continuous $(\chi,m)-\alpha$-subharmonic functions the wedge product  is always well-defined. Theorem~\ref{thm:cont-sol} is valid for a general Hermitian metric $\alpha$. The remaining issue is to verify the continuity of the envelope $u_0$. So far we could not do this for a general Hermitian metric $\alpha$.
\end{remark}

\begin{remark} Let us consider $m= n$ and $f \equiv 0$ in connection with the geodesic equation studied notably by Semmes \cite{sem92}, Donaldson \cite{don99}, Chen \cite{chen00} and Blocki \cite{blocki12}. It follows from the comparison principle (an extension of Lemma~\ref{lem:wcp-c} for $M$ in the place of $B$), that there  exists at most one continuous solution to the equation. Guan and Li  \cite{GL10} have extended the gradient estimate in \cite{blocki09} to this case. Hence, we can get a continuous solution to the homogeneous equation by a compactness argument. This solution is maximal on $M$, thus equal to $u_0$. Thus, we get the unique solution even in the case the background metric is only Hermitian.
\end{remark}

\subsection{Envelope of bounded subsolutions} \label{sec:bbd-env}

In this section we shall prove Theorem~\ref{thm:intro-3}, where
$\alpha$ is locally conformal K\"ahler. 
First we enlarge the class $\cS$ above,
\[ \notag
\hat \cS := \{v \in SH_{\chi,m}(\alpha) \cap L^\infty(\bar M): \chi_v^m\wed \alpha^{n-m} \geq \mu,\; v^*_{|_{\d M}} \leq 0 \}.
\]
The locally conformal K\"ahler assumption of $\alpha$ allows us  to use potential theory which has been developed in Section~\ref{sec:2} for bounded $(\chi,m)-\alpha$-subharmonic functions. Set
\[ \notag
	 u(z) := \sup_{v\in \hat \cS}v(z)
\]
It follows from Proposition~\ref{prop:closure-max}-(b) and  Proposition~\ref{prop:ck-ineq} that $ u^* \in \hat\cS$.  Hence, $u=u^*$.  Let us solve the linear PDE
\[ \notag \begin{aligned}
	(\chi + dd^c \rho_1) \wed \alpha^{n-1} =0, \\
	\rho_1 = 0 \quad\mbox{on } \d M.
\end{aligned}
\]
Therefore, $0\leq u \leq \rho_1$. It implies that $u = 0$  and it is continuous on  $\d M$. 

\begin{remark} It is obvious that $u_0 \leq  u$. If we can show that $ u$ is a continuous on $M$, then $ u\in \cS$ automatically. Then,  $u_0 =  u$ is indeed continuous.
\end{remark}

In what follows, we shall prove that $u$ is a solution to the (bounded) Dirichlet problem, and then we will prove its regularity by using the {\em a priori} estimate (Theorem~\ref{thm:a-priori-estimate}).

\begin{lem}[lift] Let $v \in \hat\cS$. Let $B \subset M$ be a small ball. There exists $\tilde v\in \hat\cS$ such that $v \leq \tilde v$ and $\chi_{\tilde v}^m \wedge \alpha^{n-m} = \mu$ in $B.$
\end{lem}

\begin{proof} Choose $C^0(\d B)\ni \phi_j \searrow v$ on $\d B$ and solve the Dirichlet problem 
\[ \notag\begin{cases}
v_j \in SH_{\chi,m}(\alpha) \cap C^0(\bar B), \\
(\chi + dd^c v_j)^m \wedge \alpha^{n-m} = \mu, \\
v_j = \phi_j \mbox{ on } \d B.
\end{cases}
\]
It follows from Corollary~\ref{cor:do-prin} that $v_j \searrow  w \in SH_{\chi,m}(\alpha,B)$. Hence, Theorem~\ref{thm:BT-convergence} gives that
\[ \notag
	(\chi + dd^cw)^m \wedge \alpha^{n-m} = \mu.
\]
Furthermore, $\limsup_{z \to \zeta \in \d B} w(z) \leq v(\zeta)$. By the domination principle (Corollary~\ref{cor:do-prin}) we have  
$
	v_j \geq v
$ on $B$.  
Thus, $w \geq v$ on $B$. Define 
\[ \notag \tilde v = \begin{cases} \max{\{w, v\}} \quad \mbox{on } B, \\
	v \quad \mbox{on } \bar M \setminus B.
\end{cases}
\]
Then, $\tilde v$ is the function we are looking for.
\end{proof}

\begin{lem} \label{bdd-sol}$ u \in SH_{\chi,m}(\alpha) \cap L^\infty( M) \cap C^0(\d M)$ and $\chi_u^m \wedge \alpha^{n-m} = \mu.$
\end{lem}

\begin{proof} 
It only remains to show that $\chi_u^m \wedge \alpha^{n-m} = \mu$. Fix a small ball $B \subset M$ and consider the lift $\tilde u \in \hat\cS$ of $u$ with respect to this ball. Then, $u \leq \tilde u$ in $B$. On the other hand by the definition of $u$ we have $\tilde u \leq u$. Thus, $u = \tilde u$ in $B$. Since $B$ is arbitrary, $\chi_u^m \wedge \alpha^{n-m} =\mu$ on $M$. 
\end{proof}

We shall prove the most technical part.
\begin{lem} \label{lem:continuous} $u$ is continuous on $\bar M$.
\end{lem}

By Lemma~\ref{bdd-sol}, the  function $u$ satisfies the (bounded) Dirichlet problem:
$$	\begin{aligned}
& w \in SH_{\chi,m}(\alpha) \cap L^\infty(M), \\
&(\chi + dd^c w) \wed \alpha^n = \mu,\\
&\lim_{z\to \zeta} w(z) = 0 \quad \mbox{for every } \zeta\in \d M. 
\end{aligned} $$

\begin{proof}[Proof of Lemma~\ref{lem:continuous}] We follow closely \cite[Section 2.4]{kol98}. We argue by contradiction. Suppose $u$ is not continuous, then the discontinuity of $u$ occurs at an interior point of $M$. Hence
$$ 	d = \sup_{\bar M} (u - u_*) >0,
$$
where $u_*(z) = \lim_{\epsilon \to 0} \inf_{w\in B(z, \epsilon)} u(w)$ is lower regularisation of $u$. Consider the closed nonempty set
$$ F= \{u- u_* = d\} \subset\subset M.
$$
One remark is that $u_{|_F}$ is continuous on $F$. Therefore, we may choose a point $x_0\in F$ such that
$$	u(x_0) = \min_{F} u.
$$
Choose a local coordinate chart about $x_0$, relatively compact in $M$, which is isomorphic to a small ball $B:= B(0,r) \subset \bC^n$ with origin at $z(x_0) = 0$ and of small radius.
Since $\chi \in \Gamma_m(\alpha)$, there exists $\delta>0$ such that 
\[ \label{eq:cont-eq-gamma}
	\gamma(z):=\chi(z) - \delta dd^c|z|^2 \in \Gamma_m(\alpha)
\]
for every $z \in \bar B$. Set
$$	v:= u + \delta |z|^2 \in SH_{\gamma,m}(\alpha).
$$
Since $u \geq 0$ on $M$,  $ v_*(0)= u_*(0)  \geq 0.$
Hence, we have $v\in L^\infty(\bar B)$,
which solves
\[ \label{eq:cont-eq-v}
	(\gamma + dd^c v)^m \wed\alpha^{n-m} = \mu.
\]
We also find that $$\sup_{\bar B} (v - v_*) = \sup_{\bar B}(u-u_*) = u(0) -u_*(0) =d.
$$

Let us consider the sublevel sets, for $0<s<d$, 
\[\begin{aligned}
\label{eq:level-set-near-singular}
	E(s) 
		&= \{u_* \leq u - d +s\} \cap \bar B.
\end{aligned}
\]
It's clear that $E(s)$ is closed and  by our assumption $0\in E(s)$. Furthermore,
$$	E(s) \searrow E(0)= \{u_*= u - d\} \cap  \overline{B(0,r)} \ni 0.
$$
Let us denote 
$$
	\tau(s) = u(0) - \inf_{E(s)} u(z).
$$
Since $E(s)$ is decreasing, it follows that $\tau(s)$ decreasing as $s\searrow 0$. Moreover, $\tau(s)$ is bounded for $0\leq s \leq d$. We also need the following fact.

\begin{claim}
\label{cl:first-osciliation}
	$\lim_{s \to 0} \tau(s) = 0.$ 
\end{claim}
\begin{proof}[Proof of Claim~\ref{cl:first-osciliation}] It is easy to see  that
$
	\liminf_{s\to 0} \tau(s) \geq \tau(0) =0.
$
It is enough to show that $\limsup_{s\to 0} \tau(s) \leq 0$.
Suppose that it is not true, i.e.,
$$	\notag\limsup_{s \to 0} \tau (s) = 2 \epsilon >0
$$ 
for some $\epsilon>0$. Then, there exists a sequence $s_j \to 0$ such that  $\tau (s_j) > \epsilon$ for every integer $j>0$. It means that
$$	\inf_{E(s_j)} u < u(0) - \epsilon. 
$$
Therefore, there is a sequence $\{z_j\}_{j\geq 1} \subset E(s_j)$ satisfying
$
	u(z_j) < u(0) - \epsilon.
$
Since any limit point $z$ of $\{z_j\}_{j\geq 1}$ belongs to $E(0)$,  $u(z) \geq u(0)$. Hence,
\[ \notag
	\limsup_{j \to \infty} u(z_j)  \leq u(0) - \epsilon	\leq u(z) - \epsilon.
\]
The upper semicontinuity of  $-u_*$ gives
\[ \notag
	\limsup_{j \to +\infty} [- u_*(z_j)] \leq - u_*(z).
\]
Hence,
$
	d = \limsup_{j \to +\infty} [u(z_j) - u_*(z_j)] 
	\leq u(z) - \epsilon - u_*(z) 
	= d -\epsilon. 
$
This is not possible and the claim follows.
\end{proof}

Take $B'= B(0,r')$ with a bit larger $r'>r$. By the approximation property in a small ball (Theorem~\ref{thm:approximation}), one can find a sequence 
\[\label{eq:approximant-v}	
SH_{\gamma,m}(\alpha) \cap C^\infty(B') \ni v_j \searrow v= u+\delta |z|^2 \quad \mbox{ in } B'.
\]
Let us fix this sequence from now on. If there is no otherwise indication then   $v$ and $v_j$'s are these functions. The following result is a variation of the Hartogs lemma (Lemma~\ref{lem:hartogs}).

\begin{lem}
\label{hartogs-v}
Let $K\subset \bar B$ be a compact set and $c \geq 1$ a constant. 
Assume that for some $t >0$,
$$	v < c \, v_* + t \quad \mbox{on } K.
$$
Then 
$$	v_j < c \, v + t \quad \mbox{on } K
$$
for $j>j_0$ with a fixed $j_0>0$ depending only on $K, t$.
\end{lem}

\begin{proof}[Proof of Lemma~\ref{hartogs-v}]
Let $z_0 \in K$. It follows from the assumption that
$
	z_0 \in \{v - c \, v_* < t\}
$
which is an open set by the upper semicontinuity of $v - c \, v_*$. Thus, 
$
	z_0 \in \{ v - c \, v_* < t'\}
$
for some $0< t'<t$. Hence,
$
	v(z_0) - c \, v_*(z_0) < t',
$
i.e., by definition
$$	\lim_{\epsilon' \to 0} \left( \sup_{B(z_0, 2 \epsilon')} v - c\, \inf_{B(z_0,2 \epsilon')} v \right)
	< t'.
$$
Therefore, for $0< t_1 = \frac{t-t'}{2}$, there exists 
$\epsilon'= \epsilon'(t_1, z_0)>0$ such that 
$$	B(z_0, 2\epsilon') \subset \{v < v_* + t\},
$$
and
$
	\sup_{B(z_0,2\epsilon')} v - c \, \inf_{B(z_0,2\epsilon')} v \leq t' + t_1.
$
It implies that
$$	\sup_{\bar B(z_0, \epsilon')} v \leq c \, v + t' + t_1 \quad \mbox{on } \bar B(z_0, \epsilon').
$$
By Hartogs' lemma for $(\gamma,1)-\alpha$-subharmonic functions (Corollary~\ref{cor:hartogs2}),
$$	v_j \leq \sup_{\bar B(z_0,\epsilon')}v + t_1
	<	c\, v + t' + 2 t_1 = c\, v + t.
$$
for $j \geq j(t_1, z_0, \epsilon')$. Because $K$ is compact it is covered by a finite many balls $B(z_j, \epsilon_j')$. Thus, the proof follows.
\end{proof}

We wish to apply Theorem~\ref{thm:a-priori-estimate} for the function $v$ and its approximants $v_j's$ defined in \eqref{eq:approximant-v} to get a contradiction.  Therefore, we need to study the value of $v$ and $v_j$'s on the boundary $\d B$. More precisely, we are going to show that there exists $c>1$, $a>0$ and $s_0$, which are independent of $j$, such that
\[\label{eq:set-construction}
	 \{ c \, v + d - a + s < v_j \}
\]
is non-empty and relatively compact in $B=B(0,r)$ for every $0< s<s_0 $. For this purpose we need to analyse the value of the function $c \, v - v_j$ on the boundary $S(0,r)$ of $B(0,r)$, with the help of Lemma~\ref{hartogs-v}.

Take two parameters $c>1$ and $0 <a<d$ which are determined later. 
We need to estimate \[ \notag c\, v +d -a - v_j\] on $S(0,r)$. Recall that $v = u+\delta |z|^2$ and
$$ \begin{aligned}
	E(s) &= \{u_* \leq u - d +s\} \cap \overline{B(0,r)} \\
		&= \{v_* \leq v - d +s\} \cap \overline{B(0,r)} .
\end{aligned} $$
We consider two cases:

{\bf Case 1:}  $z\in S(0,r) \cap E(a)$. We have 
\[ \notag \begin{aligned}
	v_*(z) &= u_*(z) + \delta r^2 \\
	&\geq u(z) - d + \delta r^2  \\
	&= (u(z) - u(0))  + (u(0) - d) + \delta r^2.
\end{aligned}
\]
As $0\in E(a)$, we have 
$
	\tau(a)\geq u(0) - u(z).
$
Combining with $u(0)-u_*(0) =d$,  we get that
$$	v_*(z) \geq v_*(0) - \tau(a) + \delta r^2.
$$
Note that $r>0$ (small) is already fixed. It implies that, for $c>1$,
$$	v(z) \leq v_*(z) + d < c \, v_*(z) + d  - (c-1)\left[v_*(0) + \delta r^2 - \tau(a)\right].
$$
Since $v - cv_*$ is upper semicontinuous, 
$$	v < c \, v_* + d -  (c-1)\left[v_*(0) + \delta r^2 - \tau(a)\right]
$$
on the closure of a neighbourhood $V$ of $S(0,r) \cap E(a)$. Applying Lemma~\ref{hartogs-v} for the compact set $\bar V\cap \bar B$ and 
\[ \label{eq:choose-c} 
t:=d- (c-1)\left[v_*(0) + \delta r^2 - \tau(a)\right]>0,\] we get
\[ \label{eq:boundary-point-1}
	v_j < c \, v + d -  (c-1)\left[v_*(0) + \delta r^2 - \tau(a)\right] 
	\quad \mbox{on } \bar V \cap \bar B,
\]
if $j>j_1( V)$.

{\bf Case 2:}  $z\in S(0,r)\setminus V$.
Since $E(a) \cap (S(0,r) \setminus V) = \emptyset$, the inequality
$$	v < v_* + d - a
$$
holds on $S(0,r)\setminus V$. Applying Lemma~\ref{hartogs-v} again, we get
\[ \label{eq:boundary-point-2}
	v_j < v + d - a < c \, v + d - a
	\quad \mbox{on } S(0,r) \setminus V
\]
for $j > j_2 (V)$.
Thus, it follows from \eqref{eq:boundary-point-1} and \eqref{eq:boundary-point-2} that
\[ \label{eq:condition-1}
	v_j < c\, v + d - \min\left\{a,  (c-1)\left[v_*(0) + \delta r^2 - \tau(a)\right]\right\}
\]
on $S(0,r)$ for $j> \max\{j_1, j_2\}$. 

Next, if there exists $c>1$ such that for $0< s_0 <a$, 
\[ \label{eq:condition-2}
	(c-1) v_*(0) < a - s_0
\]
then 
$
	c \, v_*(0) + d - (a-s_0) < v(0) \leq v_j(0).
$
It follows that the set 
$
	\{c \, v + d - a +s < v_j\} 
$
is non-empty for $0< s< s_0$.

According to Claim~\ref{cl:first-osciliation}, \eqref{eq:choose-c}, \eqref{eq:condition-1} and \eqref{eq:condition-2} we need to choose $0<a<d$, $c>1$ and $0<s_0<a$, in this order, 
such that 
$$\begin{aligned}
&	\tau(a) \leq \frac{\delta r^2}{2}; \\
&	d- (c-1)\left[v_*(0) + \delta r^2 - \tau(a)\right]>0; \\
&	(c-1) v_*(0) < a < (c-1)\left(v_*(0) + \frac{\delta r^2}{2}\right);\\
&	s_0 = \frac{a- (c-1)v_*(0)}{2}>0.
\end{aligned}$$
This is always possible. Thus, we get relatively compact subsets that satisfy \eqref{eq:set-construction}.

Now we can apply Theorem~\ref{thm:a-priori-estimate} to get that a contradiction. In fact, 
we have for $w_j := v_j/c$ and $0<s<s_0$,
\[ \notag
	\{c \, v + d - a +s < v_j\} = \{v + (d-a+s)/c <  w_j \} \subset\subset B.
\] 
It follows that
$$	d_j := \sup_{ B}(w_j - v) \geq \frac{d-a+s_0}{c}>0.
$$
We denote for $0<s< \vepsilon_0 < \vepsilon$ (as in Theorem~\ref{thm:a-priori-estimate}),
$$	U_j(\vepsilon,s):= \{v < (1-\vepsilon) w_j + \inf_{\Omega}[v- (1-\vepsilon)w_j]+s\}. 
$$
Notice that $\vepsilon_0$ depends only on $d, a, s_0$. Hence, applying 
Theorem~\ref{thm:a-priori-estimate} for $v$ in \eqref{eq:cont-eq-v} and $\gamma$ in \eqref{eq:cont-eq-gamma}, we get  that for $0< s <\vepsilon_0$, 
\[ \notag
	s \leq C (1+\|v\|_{L^\infty}) \|f\|_{L^p}^\frac{1}{m} [V_\alpha(U_j(\vepsilon, s ))],
\]
where $V_\alpha(U_j(\vepsilon,s)) = \int_{U_j(\vepsilon,s)} \alpha^n.$ Furthermore, for such a fixed $s >0$,
\[ \notag
	U_j(\vepsilon,s) \subset \{ v < w_j - d_j +  \vepsilon \|w_j\|_{L^\infty} + s\} \subset \{v < v_j\}.
\]
Since $V_\alpha(\{v<v_j\}) \to 0$ as $j \to +\infty$, we get the contradiction.
The proof of Lemma~\ref{lem:continuous} is finished.
\end{proof}

\subsection{Some applications} \label{sec:applications}

The first application is the mixed type inequality for Hessian operators with the Hermitian form. When both $\chi$ and $\omega$ are K\"ahler metrics the inequality is proved by Dinew and Lu \cite{chinh-dinew}. Since the inequality is local, we state it for a small Euclidean ball $B$ in $\bC^n.$

\begin{prop} \label{mix-ineq} Let $f,g \in L^p(B)$, $p>n/m$. Suppose that $u, v \in SH_{\chi,m}(\alpha) \cap C^0(\bar B)$ satisfy
\[
	\chi_u^m \wedge \alpha^{n-m} = f \alpha^n, \quad
	\chi_v^m \wedge \omega^{n-m} = g \alpha^n.
\]
Then, for any $0\leq k \leq m$, 
\[
	\chi_u^k \wedge \chi_v^{m-k} \wedge \alpha^{n-m} 
\geq 	f^\frac{k}{m} g^\frac{m-k}{m} \alpha^n.
\]
\end{prop}

\begin{proof} It is a simple consequence of the mixed type inequality in the smooth case, and then for continuous functions we use Theorem~\ref{thm:intr-1} and Lemma~\ref{lem:stability1}. 
\end{proof}

Thanks to this type of inequality with $\chi = \alpha =\omega$ we are able to relax the smoothness assumption on potentials in the statement of \cite[Proposition 3.16]{KN3}. In particular, the uniqueness of continuous solutions to the complex Hessian equation on compact Hermitian manifolds with strictly positive right hand side in $L^p$, $p>n/m$.

\begin{cor} Let $(X,\omega)$ be a compact Hermitian manifold. Suppose that $u,v \in SH_m(\omega)\cap C^0(X)$, $\sup_X u=\sup_X v =0$, satisfy
\[
	\omega_u^m \wed \omega^{n-m} =f \omega^n, \quad
	\omega_v^m \wed \omega^{n-m} =g \omega^n,
\]
where $f,g \in L^p(X,\omega^n)$, $p>n/m$. Assume that 
\[
	f \geq c_0>0
\]
for some constant $c_0$. Fix $0< a < \frac{1}{m+1}$. Then,
\[
	\|u -v\|_{L^\infty} \leq C \|f-g\|_{L^p}^a,	
\]
where the constant $C$ depends on $c_0,a,p, \|f\|_{L^p}, \|g\|_{L^p}, \omega, X$.
\end{cor}

We can also show that continuous solutions obtained in \cite{KN3} are also the continuous solutions in the viscosity sense and vice versa (Lu \cite{chinh13a} proved the existence and uniqueness of viscosity solutions to the complex Hessian equation on some special compact Hermitian manifolds).  The viscosity approach for the Monge-Amp\`ere equation on K\"ahler manifolds was used by Eyssidieux, Guedj, Zeriahi \cite{EGZ11}, Wang \cite{ywang}. It seems to be interesting to investigate the viscosity method for the complex Hessian equation on compact Hermitian manifolds with or without boundary. We refer the readers to \cite[Example 18.1]{HLa}, \cite[Example 3.2.7]{HLb} for some results in this direction.

\section{Proof of Theorem~\ref{thm:intr-1}} \label{sec:proof-main-thm} 

In this section we proceed to prove Theorem~\ref{thm:intr-1}, which we used in Sections~\ref{sec:2}, \ref{sec:3}. The proof is independent of results in those sections.

Let us rewrite the equation in the PDE form as in the paper by Sz\'ekelyhidi \cite{szekelyhidi15}. Without loss of generality we fix $\Omega :=B(0, \delta) \subset B(0,1) \subset \bC^n$  for $0<\delta <<1$. Let $\alpha$ be a Hermitian metric in $B(0,1)$. Fix a smooth real $(1,1)$-form $\chi$ on $B(0,1)$.
For a $C^2$ function $u$ we consider the real $(1,1)$-form $g = \chi + \ddbar u$, i.e., $g_{i \bar j} = \chi_{i \bar j} + u_{i \bar j}$. We can define
$A^{i}_j := \alpha^{\bar p i} g_{j\bar p}$, where $\alpha^{\bar j i}$ is the inverse of $\alpha_{i\bar j}$. Then, the matrix $A^{i}_j$ is Hermitian with respect to the metric $\alpha$, i.e., $A\times [\alpha_{i\bar j}]$ is a Hermitian matrix. Denote 
$\lambda (A)=(\lambda_1,...,\lambda_n) \in \bR^n$ the $n$-tuple of eigenvalues of $A$. In other words, $\lambda$ is the eigenvector of $g_{i\bar j}$ with respect to the metric $\alpha$.  The complex Hessian equation \eqref{intro-dp} is 
\[ \notag
	F(A) = h,
\]
where 
\[ \notag
	F(A):= f(\lambda (A)) = [(S_m(\lambda)]^{1/m},
\]
and $f$ is a symmetric increasing concave function defined on the cone $\Gamma_m$. Recall that the $m$-th elementary symmetric cone is
\[ \notag
	\Gamma_m =\{\lambda \in \bR^n: S_1(\lambda) >0, ..., S_m(\lambda)>0\} .
\]

Fix $0<h\in C^\infty(\bar \Omega)$ and a smooth boundary data $\varphi\in C^\infty(\d \Omega)$. We wish to study the Dirichlet problem, seeking $u \in C^\infty(\bar \Omega)$ and $u= \vphi$ smooth on $ \partial \Omega$ such that
\[ \label{mdp}
\begin{cases}
	\lambda(A) &\in \Gamma_m, \\
	F(A) &=h,
\end{cases}
\]
where $A^i_j = \alpha^{\bar p i} (\chi_{j \bar p} + u_{j\bar p})$. 
To simplify notation, first we extend $\vphi \in C^\infty(\d\Omega)$ smoothly  to $\overline{B(0,1)}$. Upon replacing 
\[ \notag\begin{aligned}
	\tilde u := u - C(|z|^2 - \delta^2) - \vphi, \\
	\tilde\chi := \chi + \ddbar [C(|z|^2 - \delta^2) + \vphi], 
\end{aligned}
\]
with $C>0$ large enough, which does not change $g_{i\bar j}$, we may assume that 
\[ \label{eq: reduction}
	u = 0 \quad \mbox{on } \d\Omega, \quad \chi \geq \alpha \quad\mbox{on } \bar\Omega\] and $0$ is the subsolution, i.e., $\chi^m \wedge \alpha^{n-m} \geq h.$
	
Let $F^{ij}(A):= \d F/ \d a_{ij}$ be the partial derivative of $F$ at $A$ with respect to entry $a_{ij}$. We also denote
\[ \notag
	\cF := \sum_{1 \leq i \leq n} f_i, 
\]
where $f_i = \d f/ \d \lambda_i>0$ are precisely eigenvalues of $F^{ij}$ with respect to metric $\alpha$. If we choose coordinates in which $\alpha$ is orthonormal and $A$ being diagonal, then \[\notag F^{ij} = \delta_{ij}f_i,\] and thus $\cF= \sum_{i=1}^nF^{ii}$.

We will proceed in Sections~\ref{sec:c0}, \ref{sec:c1}, \ref{sec:c2} to get {\em a priori} estimates, up to second order, and using the results in Tosatti-Weinkove-Wang-Yang \cite{twwy14}, to get $C^{2,\alpha}$ interior estimates. This combined with the $C^2$-estimates at the boundary thus gives the full $C^2$ estimates up to the boundary of the real Hessian of $u$. This allows us to apply Krylov's boundary estimate \cite{krylov} to get the desired $C^{2,\alpha}(\bar\Omega)$ estimate. The higher order estimates are obtained by the bootstrapping argument,  and then using the continuity method to obtain a solution to the equation \eqref{mdp}. The uniqueness follows from the maximum principle.

\section{$C^0-$estimate}
\label{sec:c0}

Denote $B^i_j = \alpha^{\bar p i} \chi_{j\bar p}$. Then, $F(A) =h  \leq F(B)$ and $u\geq 0$ on $\d\Omega$.  Solve the linear PDE
\[\notag
\begin{cases}
	n(\chi + \ddbar u_1) \wedge \alpha^{n-1}/\alpha^n &= 0, \\
	u_1&=0 \quad \mbox{on } \d \Omega.
\end{cases}
\]
By the maximum principle we get that for some $C_0>0$,
\[ \label{uni-bound}
	0\leq  u \leq u_1 \leq C_0.
\]
As $u = u_1 = 0$ on $\d\Omega$, it also follows that for some $C_0'>0,$
\[ \label{grad-boudary}
	|\nabla u| \leq C_0' \quad \mbox{ on } \d\Omega.
\]

\section{$C^1-$ estimate} 
\label{sec:c1} 

In this section we prove the gradient estimate. Here the assumption of small radius is important. (Notice that Pli\'s \cite{plis} has claimed this estimate in the case $\chi \equiv 0$ and $\alpha$ K\"ahler for any ball but no proof was given there.) 

By \eqref{eq: reduction} we may suppose that for some $C_1>0$,
\[ \label{metric-bound}
	\frac{\delta_{ij}}{C_1} \leq \alpha_{i\bar j}\leq  \chi_{i \bar j} \leq C_1 \delta_{ij}.
\]
\[ \notag
	L:= \sup_{\Omega}|u| +1.
\]
Let $\nabla$ denote the Chern connection with respect to $\alpha$. Note that $\|z\|^2_\alpha$ is strictly plurisubharmonic  as long as $\delta$ small. More precisely, we choose $\delta$ so that
\[ \label{cdn1}
	\nabla_{\bar p} \nabla_p \|z\|_\alpha^2 = \d_{\bar p} \d_p(\alpha_{i\bar j} z^i \bar z^j) = \alpha_{p\bar p} + O(|z|) \geq \alpha_{p\bar p}/2.
\]
Denote $v = N(\sup_{z\in \Omega} \|z\|_\alpha^2 -\|z\|_\alpha^2),$ where $N>0$ is a constant to be determined later. We see that \[ 0\leq v \leq NC_1 \delta^2\quad \mbox{and}\quad -v_{p\bar p}= -\d_p\d_{\bar p} v \geq N/2C_1. \]
Consider 
\[ \notag
	G = \log \|\nabla u\|_\alpha^2 + \psi (u+v),
\]
with 
\[\notag
	\psi(t) = -\frac{1}{2} \log(1+ \frac{t}{L+ NC_1 \delta^2}).
\]
Note that a similar function was considered by Hou-Ma-Wu \cite{hou-ma-wu} and it satisfies
\[ \psi'<0, \quad \psi'' = 2\psi'^2.\]

If $G$ attains its maximum at a boundary point, then $\sup_{\Omega}|\nabla u|$ is uniformly bounded by $\sup_{\d\Omega}|\nabla u|$, up to a uniform constant.  By \eqref{grad-boudary}, the latter one is uniformly bounded. Then, we will get the $C^1$-- estimate. Therefore, we may assume that the maximum point belongs to $\Omega$. We shall derive the desired estimate by using maximum principle at this point.

We choose the orthonormal coordinates for $\alpha$ such that at this point
$\alpha_{i\bar j}$ is the identity matrix and $A^i_j$ is diagonal. All computations bellow are performed at this point and the subscripts stand for usual derivatives if there is no otherwise indication.

Differentiating $G$ twice and evaluating the equations at the maximum point we have:
\[ \label{eq-dg} 
	G_p = \frac{(\nabla_p\nabla_i u) u_{\bar i} + u_i \nabla_p \nabla_{\bar i} u}{|\nabla u|^2} 
	+ \psi' (u_p + v_p) =0;
\]
\[\label{grad-eq-ddg}
\begin{aligned}
	G_{p \bar p} 
&	= \frac{(\nabla_{\bar p} \nabla_p \nabla_{i} u) u_{\bar i} + u_i \nabla_{\bar p} \nabla_p \nabla_{\bar i} u + |\nabla_p \nabla_{i} u|^2 + |\nabla_{\bar p}\nabla_{i} u|^2}{|\nabla u|^2} \\
&	\quad - \frac{1}{|\nabla u|^4} \left | u_i \nabla_p \nabla_{\bar i} u + u_{\bar i} \nabla_p\nabla_{i}u \right|^2 \\
&	\quad + \psi'' |u_p + v_p|^2 + \psi' (u_{p\bar p} + v_{p\bar p}).
\end{aligned}
\]
Next, we have
\[ \label{com3-1}
\begin{aligned}
\nabla_{\bar p}\nabla_p \nabla_i u &= u_{p\bar p i} - (\d_{\bar p}\Gamma_{pi}^q) u_q - \Gamma_{pi}^q u_{q \bar p}\\
&=g_{p\bar p i} - \chi_{p\bar p i} - (\d_{\bar p}\Gamma_{pi}^q) u_q - \Gamma_{pi}^p \lambda_p + \Gamma_{pi}^q \chi_{q\bar p},
\end{aligned}
\]
where we used that $g_{i\bar j}$ is diagonal. Similarly, 
\[ \notag
\begin{aligned}
\nabla_{\bar p}\nabla_p \nabla_{\bar i} u &= u_{p\bar p \bar i} - \overline{\Gamma_{pi}^q} u_{p\bar q} \\
&=	g_{p\bar p \bar i} - \chi_{p\bar p \bar i} - \overline{\Gamma_{pi}^p} \lambda_p +\overline{\Gamma_{pi}^q} \chi_{p\bar q}.
\end{aligned}
\]
Moreover, by applying the covariant derivatives to  the equation we get
\[\notag
F^{pp} \nabla_i g_{p \bar p} = h_i.
\]
As $\nabla_i g_{p\bar p} = g_{p\bar p i} - \Gamma_{ip}^m g_{m\bar p}$, we have
$
	F^{pp} g_{p\bar p i} = h_i + F^{pp} \Gamma_{ip}^p \lambda_p.
$
Combining with \eqref{com3-1} we get that
\[ \label{grad-eq-t1a}
\begin{aligned}
	F^{pp} (\nabla_{\bar p}\nabla_p \nabla_i u) u_{\bar i} 
&=	h_i u_{\bar i}+  F^{pp}(\Gamma_{ip}^p-\Gamma_{pi}^p) \lambda_p u_{\bar i} - F^{pp} \chi_{p\bar p i}u_{\bar i} \\
&\quad- F^{pp}(\d_{\bar p}\Gamma_{pi}^q) u_qu_{\bar i} +F^{pp} \Gamma_{pi}^q \chi_{q\bar p} u_{\bar i}.
\end{aligned}
\]
Similarly,
\[ \label{grad-eq-t1b}
\begin{aligned}	
	F^{pp} (\nabla_{\bar p}\nabla_p \nabla_{\bar i} u) u_{i} 
&=	h_{\bar i} u_i + F^{pp}(\overline{\Gamma_{ip}^p}-\overline{\Gamma_{pi}^p}) \lambda_p u_i \\
&\quad - F^{pp}\chi_{p\bar p \bar i} u_i - F^{pp} \overline{\Gamma_{pi}^q} \chi_{p\bar q} u_i
\end{aligned}
\]
Let's denote
\[ \notag
	R:=\sup_{p,q, i} |\d_{\bar p} \Gamma_{pi}^q|, \quad
	T:= \sup_{i,p} |\Gamma_{ip}^p-\Gamma_{pi}^p|,
\]
which are  bounds for the curvature and torsion of metric $\alpha$ on $\bar B(0,1)$.

It follows from \eqref{grad-eq-t1a} and \eqref{grad-eq-t1b} that, for $K:= |\nabla u|^2$ large enough, 
\[ \label{grad-eq-t12}
\begin{aligned}
&	\frac{1}{K}F^{pp} [(\nabla_{\bar p}\nabla_p \nabla_i u) u_{\bar i} + (\nabla_{\bar p}\nabla_p \nabla_{\bar i} u) u_{i} ] \\
&\geq -C/K^{1/2} - F^{pp}|\lambda_p| T/K^{1/2} - C \cF/K^{1/2} - R\cF \\
&\geq - C - \frac{1}{2K} F^{pp} \lambda_p^2 - (R+T^2+1) \cF,
\end{aligned}
\]
where in the last inequality we used
\[\notag
	\frac{|\lambda_p|T}{K^{1/2}} \leq \frac{1}{2}(\frac{\lambda_p^2}{K} +T^2).
\]
By the equation \eqref{eq-dg} 
\[ \label{grad-eq-t3}
	- \frac{1}{K^2} \left | u_i \nabla_p \nabla_{\bar i} u + u_{\bar i} \nabla_p\nabla_{i}u \right|^2 =- \psi'^2|u_p+v_p|^2.
\]
By $\sum_{p=1}^n f_p\lambda_p =h$ and $\chi_{p\bar p}\geq 1$ we have
\[ \label{grad-eq-t4}
\begin{aligned}
\psi' F^{pp}(u_{p\bar p} + v_{p\bar p}) 
&= \psi' F^{pp} \lambda_p + |\psi'| F^{pp}[\chi_{p\bar p} + (-v_{p\bar p})] \\
&\geq -C + |\psi'| [1 +N/2C_1] \cF.
\end{aligned}
\]
We also note that
\[ \label{grad-eq-t1c}
\begin{aligned}
\frac{1}{K}	F^{pp}|\nabla_{\bar p} \nabla_i u|^2 
&= \frac{1}{K} F^{pp}|g_{i\bar p}-\chi_{i\bar p}|^2 \\
&\geq \frac{1}{2K} F^{pp} |\lambda_p|^2 -\frac{1}{K} F^{pp} |\chi_{i\bar p}|^2\\
&\geq \frac{1}{2K} F^{pp} |\lambda_p|^2 - \frac{C\cF}{K}.
\end{aligned}
\]
Therefore, combining \eqref{grad-eq-ddg}, \eqref{grad-eq-t12}, \eqref{grad-eq-t3}, \eqref{grad-eq-t4} and \eqref{grad-eq-t1c}, we get that
\[ \notag
\begin{aligned}
	0\geq F^{pp} G_{p\bar p} 
&\geq - C- \frac{1}{2K} F^{pp} |\lambda_p|^2 - (R+T^2+1) \cF \\
&\quad + \frac{1}{2K} F^{pp} |\lambda_p|^2 - \frac{C\cF}{K} \\
&\quad+ (\psi'' - \psi'^2) F^{pp}|u_p+v_p|^2 \\
&\quad + |\psi'| [1 +N/2C_1] \cF.
\end{aligned}
\]
We may assume that $K>C$. As $\psi'' = 2 \psi'^2$, we simplify the inequality:
\[ \label{max-ine1}
	0\geq \psi'^2F^{pp} |u_p+v_p|^2+|\psi'| (1 + N/2C_1) \cF - (R+T^2+2)\cF -C.
\]

Now we decrease further $\delta$ (if necessary)  so that $16(R+T^2+3)C_1^2\delta^2 <1$. Hence, we can choose $N>1$ satisfying
$$
 \frac{N}{8(LC_1+NC_1^2\delta^2)} 
	\geq R+T^2+3.
$$
On the interval $t\in [0, L+ NC_1\delta^2]$, we have 
$
	|\psi'| \geq 1/4(L+NC_1\delta^2).
$
Hence, 
\[ \label{max-ine2}
	\frac{N|\psi'|}{2C_1} \geq R+T^2+3.
\]
It follows from \eqref{max-ine1} and \eqref{max-ine2} that
\[ \label{key-ineqs}
	F^{pp}|u_p+ v_p|^2 + \cF \leq C,
\]
where $C= C(A, C_1, L)$.
We shall  use \eqref{key-ineqs} to prove that
\[ \notag
 	F^{ii} = \frac{S_m^{-1+1/m} (\lambda)}{m} S_{m-1;i}(\lambda) \geq c>0
\]
for some uniform $c$ and for every $1 \leq i \leq n$. Indeed, since
\[ \notag
	\cF = \frac{S_m^{-1+1/m} (\lambda)}{m}\sum_{i =1}^n S_{m-1;i}(\lambda) \leq C,
\]
we have $S_{m-1;i}(\lambda) \leq C$ for every $i =1,...,n$. 
By the inequality \cite[Proposition 2.1 (4)]{wang96}
\[ \notag
	\prod_{i=1}^n S_{m-1;i}(\lambda) \geq C_{n,m} [S_{m}(\lambda)]^{n(m-1)/m},
\]
where $C_{n,m}>0$ depends only on $n,m$. Thus, the desired lower bound for each $S_{m-1;i}(\lambda)$ follows from the equation $(S_m(\lambda))^\frac{1}{m} = h>0$ and the upper bound for $S_{m-1;i}(\lambda)$. We also get the lower bound for each $F^{ii}$. Finally, from 
\[ F^{pp} |u_p + v_p|^2 \leq C \notag
\]
we  easily get the a priori gradient bound, $|\nabla u| \leq C$.

\section{$C^2-$ estimates} 
\label{sec:c2}

In this section we prove the following estimate
\[\label{c2-global}
	\sup_{\bar \Omega} |\ii\d\bar\d u| \leq C,
\]
where $C$ depends on $\|u\|_{L^\infty(\bar \Omega)}$, $\|\nabla u\|_{L^\infty(\bar \Omega)}$ and the given data.

If $\sup_{\bar \Omega} |\d\bar\d u|$ is attained at an interior point of $\Omega$, then by a result of Sz\'ekelyhidi \cite{szekelyhidi15} (see also Zhang \cite{dzhang15}) we have for some $C>0$, which depends on $\|u\|_\infty$ and  the given data,
\[ \notag
	|\ii\d\bar\d u| \leq C(1+ \sup_{\bar \Omega} |\nabla u|^2).
\]
Therefore, we only need to consider the case when the maximum point  $P$ is on the boundary. At this point, following Boucksom \cite{boucksom},  we choose a local half-ball coordinate $U$ such that $z(P)=0$ and $r$ is the defining function for $U \cap \d\Omega$.  Then, $U \cap \Omega = \{r \leq 0\} \cap \Omega$. We choose the   coordinates $z=(z_1,...,z_n)$, centred at $0$, such that  the positive
$x_n$ axis is the interior normal direction, and near $0$ the graph $U \cap\d \Omega$ is written as
\[ \label{def-funct}
	r = -x_n + \sum_{j,k=1}^n a_{jk} z_j \bar z_k + O(|z|^3) =0.
\] 
We refer the reader to the expository paper of Boucksom \cite{boucksom} for more details on this coordinate. 

Recall that $\lambda_i$'s are eigenvalue functions of matrix $A$, i.e.
\[ \notag
 \lambda(A)= (\lambda_1, ...,\lambda_n).
\]
We often represent quantities in the orthonormal coordinates $(w^1,..,w^n)$ in which $\alpha_{i\bar j}$ is the identity and $A^i_j$ is diagonal. The following equations will help us in computing quantities in the orthonormal coordinates once we know theirs forms in the fixed coordinates $(z^1,...,z^n)$. 

Suppose at a given point we change the coordinates, $w = Xz$, i.e. 
\[ \notag
	w^i = x_{ik}z^k, \quad x_{ik} \in \bC,
\]
and we obtain at that point \[ \notag 
\begin{aligned} 
&\alpha_{i\bar j}\ii dz^i\wedge d\bar z^j = \sum_{a=1}^n \ii dw^a\wedge d\bar w^a;\\ 
& g_{i\bar j} \ii dz^i\wedge d\bar z^j= \sum_{a=1}^n\lambda_a \ii dw^a \wedge d\bar w^a. \end{aligned}
\] It follows that
\[ \notag
	\alpha_{i\bar j} = x_{ai} \overline{x_{aj}}, \quad
	g_{i\bar j} = x_{ai} \lambda_a  \overline{x_{aj}}.
\]
It is clear that for every $1\leq i \leq n,$ 
\[ \sum_{a=1}^n |x_{ai}|^2 = \alpha_{i\bar i} < C.\notag\] Moreover, the inverse of matrix $\alpha_{i\bar j}$ is given by the formula  
\[\alpha^{\bar j i} = \overline{x^{ja}} x^{ia},\notag\]
where $x^{ia}$ is the inverse of $X$. Hence,
\[ \notag
	A^i_j = \alpha^{\bar pi} g_{j\bar p} = x^{ia} \lambda_a x_{aj}.
\]
In $\bC^{n\times n}$ if we change coordinates $B= X A X^{-1} = (b_{kl})$, then at the considered point $B$ is a diagonal matrix $(\lambda_1,...,\lambda_n)$. Therefore, $\lambda_a$ is smooth at the diagonal matrix $B$ (see e.g. \cite{sprucknote}) and 
\[\frac{\d F}{\d b_{kl}} = \frac{\d f}{\d \lambda_a} \cdot \frac{\d \lambda_a}{\d b_{kl}} = f_a \delta_{ak} \delta_{al};\notag\]
\[ \notag
 \frac{\d F}{\d a_{ij}} = \frac{\d F}{\d b_{kl}} \frac{\d b_{kl}}{\d a_{ij}}
= \sum_{k,l}\sum_{a=1}^n f_a \delta_{ak} \delta_{al} x_{ki} x^{jl}
= x^{ja} f_a x_{ai}.
\]
An easy consequence  from the above formula is that
\[ \notag
	L^{\bar pj}:= F^{ij}\alpha^{\bar p i} = \overline{x^{pa}} f_a x^{ja},
\]
where  $F^{ij} = \d F/\d a_{ij}$ at $A^{i}_j$, is a positive definite Hermitian matrix.

To derive the desired {\em a priori} estimate we will use the linearised elliptic operator, for a smooth function $w$, 
\[ \notag
	L w:= L^{\bar p j} \d_j \d_{\bar p} w = F^{ij} \alpha^{\bar p i}\d_j \d_{\bar p} w,
\] 
It is worth to recall that
\[ \notag
	\cF := \sum_{1 \leq i \leq n} f_i
\]
where $f_i = \d f/ \d \lambda_i$ are eigenvalues of $F^{ij}$ with respect to metric $\alpha$.

Following Guan \cite{guan98} (c.f. Boucksom \cite{boucksom}) we construct the important barrier function.

\begin{lem}
\label{bfct}
Set $b= u -  r - \mu r^2$. Then, there exist  constants $\mu>0$ and $\tau>0$ such that 
\[ \notag
	L b \leq -\frac{1}{2} \cF
\]
and $b \geq 0$ on the half-ball coordinate $U$ of  radius $|r|<\tau$.
\end{lem}

\begin{proof} By shrinking the radius of the half coordinate ball $U$, we have $r$ is plurisubharmonic  in $U$. Then
\[ \label{bar-lem-1}
	0 \leq L r = L^{\bar p j} r_{j\bar p}\leq C \cF.
\]
As $b_{j\bar p}:= \d_j\d_{\bar p} b$ is a Hermitian matrix and $\alpha_{i\bar j}>0$, we can represent
\[\notag	
	b_{j\bar p} = x_{aj} \gamma_a \overline{x_{ap}},
\]
where $\gamma_a \in \bR$ are eigenvalues of $b_{i\bar j}$ with respect to  the matrix $\alpha_{i\bar j}$. Hence,
\[ \notag
	Lb = \sum_{a=1}^n f_a\gamma_a 
\]
which does not depend on the choice of coordinates of $\alpha$. Thus, to verify the desired inequality at a given point, we compute, at this point, in orthonormal coordinates of $\alpha$ and $A^i_j = \alpha^{\bar p i}(\chi_{j\bar p} + u_{j\bar p})= (\lambda_1,...,\lambda_n)$  diagonal. So is $L^{\bar p i} = (f_1,...,f_n)$. 

We now compute, as $r\leq 0$, 
\begin{align} 
Lb 
&= 	L^{\bar i i} u_{i \bar i} -  Lr -2 \mu r Lr -   2 \mu L^{\bar i i} |r_{i}|^2  \notag\\ 
&=	L^{\bar i i} g_{i \bar i} - L^{\bar i i} \chi_{i \bar i} -  Lr 
+ 2 \mu |r| Lr - 2 \mu L^{\bar i i} r_{i}^2  \notag\\ 
&=	\sum_{i=1}^n f_i \lambda_i  + (2 \mu |r|-1) Lr - 
L^{\bar i i} (\chi_{i\bar i} + 2 \mu r_i^2). \label{eq-bar}
\end{align}
We have
$	\sum_{i=1}^n f_i \lambda_i = h $ and 
\[ \label{bar-lem-t2}
	(2 \mu |r|-1) Lr \leq 2C\mu |r| \cF.
\]
Notice that $\chi_{i\bar i} \geq \alpha_{i\bar i} =1$. The last negative term \eqref{eq-bar} will be divided into three parts. First
\[
	- L^{\bar i i} \chi_{i\bar i}/2 \leq - \cF/2. \notag
\]
Next, we use $-L^{\bar i i} \chi_{i\bar i}/4$ to absorb the right hand side of \eqref{bar-lem-t2} (i.e. the second term in \eqref{eq-bar}), provided that
\[ \notag
	C\mu |r| \leq 1/8.
\]
We will use the part $- L^{\bar i i}(\frac{\chi_{i\bar i}}{4} +2 \mu r_i^2)$ for $\mu$ large to absorb the first term in \eqref{eq-bar}. We claim that 
\[ \label{bar-lem-t3}
	L^{\bar i i} (\frac{\chi_{i\bar i}}{4} +2 \mu r_i^2) \geq c_0 \mu^\frac{1}{m}
\]
for some uniform $c_0>0$. In fact, if $m=1$, then it is obvious. We may assume that $m>1$. Observe that $|\nabla r|>0$ at $0$, then decrease $\tau$ if necessary, we have  
\[ \label{bar-lem-2}
	|\nabla r|^2 = \sum_{i=1}^n r_i^2 > c_1
\]
for a uniform $c_1>0$ on $U$. By G\aa rding's inequality \cite[Theorem~5]{garding59} with $\lambda' = (\frac{\chi_{1\bar 1}}{4} + 2\mu r_1^2, ..., \frac{\chi_{n\bar n}}{4} +2 \mu r_{n}^2\,)$ and $S_{m-1;i}(\lambda)$, we have
\begin{align*}
	\sum_{i=1}^n (\frac{\chi_{i \bar i}}{4} + 2\mu |r_i|^2)S_{m-1;i}(\lambda) 
&\geq 	m [S_{m}(\chi_{i\bar i}/4 +2 \mu |r_i|^2)]^\frac{1}{m} [S_m(\lambda)]^\frac{m-1}{m} \\
&\geq \frac{m \mu^\frac{1}{m} h^{m-1}}{4^\frac{m-1}{m}} \left(\sum_{i=1}^n 2|r_i|^2 \prod_{k \neq i} \chi_{k\bar k}\right)^\frac{1}{m} \\
&\geq \frac{m \mu^\frac{1}{m} h^{m-1}}{4^\frac{m-1}{m}} \left(\sum_{i=1}^n 2|r_i|^2\right)^\frac{1}{m} \\
& \geq \frac{2^\frac{1}{m}m \mu^\frac{1}{m} h^{m-1}}{4^\frac{m-1}{m}}  c_1^\frac{1}{m},
\end{align*}
where we used $\chi_{k\bar k}\geq 1$ for the third inequality and used \eqref{bar-lem-2} for the last inequality. 

To obtain the inequality \eqref{bar-lem-t3}, we only need to notice that 
\[\notag
	L^{\bar i i} = f_i = \frac{[S_{m}(\lambda)]^{(1-m)/m}S_{m-1;i}(\lambda)}{m}. 
\]
Therefore, the uniform constant we get is $c_0 = C(c_1, h,m)>0$.
So we can choose $\mu>0$ large enough to get the desired inequality for $Lb$. 

It remains to check that $b\geq 0$. Since $u \geq 0$ it is enough to have that
\[\notag
	-  r - \mu r^2 = |r|(1 - \mu |r|) \geq 0.
\]
This easily follows by further decreasing (if necessary) the radius $\tau$ of the half-ball coordinate.
\end{proof}

We are ready to prove the second order estimates for $u$ at the boundary point $0 \in \d \Omega$. Following Caffarelli, Nirenberg, Kohn, Spruck \cite{CKNS82} (c.f \cite{boucksom}) we set \[ t_1 = x_1, t_2 = y_1, ..., t_{2n-2}= y_{n-1}, t_{2n-1} = y_{n}, t_{2n} = x_{n}.\notag\] 
Let $D_1, ..., D_{2n}$ be the dual basis of $dt_1, ..., dt_{2n-1}, - dr, \notag$ then 
\[ \notag
	D_j = \frac{\d}{\d t_j} - \frac{r_{t_j}}{r_{x_n}} \frac{\d}{\d x_n} \quad \mbox{ for }
	\quad 1\leq j < 2n,
\]
and \[ D_{2n} = - \frac{1}{r_{x_n}} \frac{\d}{\d x_n}. \notag\]

Because $u = 0$ on $\d \Omega$, we can write, for some positive function $\sigma,$
\[\notag
	u = \sigma r.
\]
Then,
\[ \label{sigma-0}
	\d u/ \d x_n(0) = - \sigma(0).
\]
So, $|\sigma(0)|<C$. Moreover, for $1 \leq j \leq 2n-1$,
\[ \notag
	\frac{\d^2 u}{\d t_i \d t_j}(0) = \sigma(0) \frac{\d^2 r}{\d t_i \d t_j}(0)
\]
and hence tangential-tangential derivatives $|\d_{t_i}\d_{t_j} u|$ are under control.

Next, we bound normal-tangential derivatives:
\begin{thm} \label{nt-bound}
We have 
\[ \notag
	\left|\frac{\d^2 u}{\d t_j \d x_n}(0)\right| \leq C \quad \mbox{for}\quad j \leq 2n-1,
\]
where $C$ depends on $u, |\nabla u|$ and the given datum.
\end{thm}

\begin{proof} Without loss of generality we fix $j =1$ and we shall show that 
\[ \notag
	|D_{2n} D_1 u(0)| \leq C.
\]
The derivative $D_1$, acting on functions, is equal to 
\[ \notag
	\partial_1 + \partial_{\bar 1} + \tilde r (\partial_n + \partial_{\bar n}),
\]
where $\partial$ denotes the usual partial derivatives and 
$ \tilde r := - \frac{r_{x_1}}{r_{x_n}}$ is a smooth real-valued function near $0$. 
Recall that we use the subindex to denote usual derivatives in direction $\d/\d z_1,...,\d/\d z_n$ and their conjugates if there is no other indication. This gives
\[ \notag
	D_1u = u_1 + u_{\bar 1} + \tilde r (u_n + u_{\bar n}).
\]

Following Caffarelli, Nirenberg, Spruck \cite{CNS85} and Guan \cite{guan98,guan14}, our goal is to construct a function of form
\[ \notag
\begin{aligned}
	w &=D_1u  - \sum_{k<n} |u_k|^2 -  |u_n - u_{\bar n}|^2 + \mu_1 b+ \mu_2 |z|^2, \\
\end{aligned}
\]
satisfying the following:
\begin{itemize}
\item[(i)]
$w(0) =0;$
\item[(ii)]
$w \geq 0$ on  $\d U$; 
\item[(iii)]
$L w = L^{\bar p j}\partial_j \partial_{\bar p} w \leq 0 $ in the interior of $U$,
\end{itemize}
where $b$ is the barrier function constructed in Lemma~\ref{bfct}, constants $\mu_1, \mu_2>0$ are to be determined later.

To see the first property $(i)$ we note that, for $i<n$,
\[ \notag
	2u_i(0) = \frac{\d u}{\d x_i}(0) - \ii \frac{\d u}{\d y_i}(0) =0,
\] 
and 
\[ \notag
	u_n(0) - u_{\bar n}(0)= -\ii \frac{\d u}{\d y_n}(0) =0.
\]
Moreover, $D_1u(0)=b(0) = \tilde r (0) =0$. Therefore, the first property follows.

Next, we verify the second property $(ii)$. We claim that there exists a constant $\mu_2>0$  such that
\[\notag
	w \geq 0 \quad \mbox{on} \quad \d U.  
\]
To see this  consider two parts $\d \Omega \cap U$ and $\d U \setminus (\d\Omega  \cap U)$ of the boundary $\d U$ separately.

{\bf Part 1:} On $\d \Omega \cap U$. We know that  
$D_1u = b = 0$, and near $0$
\[ \notag
	x_n = \sum_{j,k=1}^n a_{jk} z_j \bar z_k + O(|z|^3).
\] 
 By writing $x_n = \rho (t_1,...,t_{2n-1}) = \rho(t)$ we deduce that 
\[ \notag
	\rho(t) = \sum_{i,j <2n} k_{ij} t_i t_j + O(|t|^3),
\]
where $(k_{ij}) = \left[\frac{\d^2 x_n}{\d t_i \d t_j}(0) \right]$ is uniformly bounded. Since $u(t, \rho(t)) =0$, 
\[ \notag
	\d u/\d t_i + \d u/ \d x_n \cdot \d \rho/ \d t_i =0
\]
for $i<2n$. Applying for $y_n = t_{2n-1}$ gives
\[ \notag
 	 |\d u/\d y_n|^2 \leq C |t|^2 \leq C |z|^2. 
\]
Similarly, for $i<n$,
\[ \notag
	|u_i|^2 \leq C|z|^2.
\]
Therefore, $w\geq 0$ on $\d\Omega \cap U$ for $\mu_2>0$ large enough.

{\bf Part 2:} On $\d U \setminus (\d\Omega \cap U)$. On this piece  $|z|^2 = \tau^2$ with $\tau$ being the radius of $U$. Since $b \geq 0$ on $U$, we have $w\geq 0$ as soon as
\[ \notag \mu_2 \tau^2 \geq |D_1 u| + \sum_{i=1}^n |u_i|^2. 
\] 
This is done by choosing $\mu_2>0$ large as the right hand side is under control by the $C^1-$estimate. Thus, the second property is satisfied.

To verify the third property $(iii)$,  $L^{\bar p j} w_{j\bar p} \leq 0$ in the interior of $U$, we fix an interior point  $z_0 \in U$. 
Below we compute at this fixed point. The estimation will be split into several steps. 

{\bf (1) Estimate for $D_1u$.} We start by computing
\[ \label{cp-f1}
\begin{aligned}
	L^{\bar p j} (D_1u)_{j\bar p} 
&= 	L^{\bar p j} \left[ u_1 + u_{\bar 1} + \tilde r (u_n+ u_{\bar n})\right]_{j\bar p} \\
&=	 L^{\bar p j} [u_{1 j \bar p} + u_{\bar 1 j \bar p} + \tilde r (u_{n i \bar p} + u_{\bar n j \bar p})] \\
&\quad + L^{\bar p j} [\tilde r_j (u_{n}+ u_{\bar n})_{\bar p} + \tilde r_{\bar p} (u_n + u_{\bar n})_j] \\
&\quad + L^{\bar p j} \tilde r_{j \bar p}(u_n + u_{\bar n}) \\
&=: I_1 + I_2 + I_3.
\end{aligned}
\]
Let us denote $K:= \sup_{\Omega}|\nabla u|^2$, which is bounded by the $C^1-$estimate.

\begin{lem} \label{i2-main} There exists a constant $C$ depending only on $\alpha$ such that for any fixed $j,q$,
\[ \notag
	|L^{\bar p j} g_{q\bar p}| \leq C \sum_{i=1}^n f_i |\lambda_i|.
\]
Similarly,
\[ \notag
	|L^{\bar j p} g_{p\bar q}| \leq C \sum_{i=1}^n f_i |\lambda_i|.
\]
\end{lem}

\begin{proof} Recall that we have $
	\alpha_{i\bar j} = x_{ai} \overline{x_{aj}}, $ $
	g_{i\bar j} = x_{ai} \lambda_a  \overline{x_{aj}}, $
and $L^{\bar pj} = \overline{x^{pa}} f_a x^{ja}.$ Therefore,
\[\notag
	L^{\bar p j} g_{q\bar p} = \overline{x^{pa}} f_a x^{ja} x_{bq}\lambda_b \overline{x_{bp}} = x^{ja} f_a \lambda_a x_{aq}. 
\]
Thus, the conclusion follows. The second inequality is proved in the same way.
\end{proof}

{\bf (1a) Estimate $I_2$ and $I_3.$}  We first easily have
\[ \label{est-i3}
\begin{aligned}
	|I_3| = |L^{\bar p j} \tilde r_{j \bar p}(u_n + u_{\bar n})| 
&\leq	C K^\frac{1}{2} \cF \\
&\leq C \cF.
\end{aligned}
\]
Since two terms in $I_2$ are conjugate, so we will estimate one of them. We proceed as follows:
\[ \notag
\begin{aligned}
	\tilde r_j (u_n + u_{\bar n})_{\bar p} 
&=	\tilde r_j[2 u_n -(u_n- u_{\bar n})]_{\bar p} \\
&=	2  \tilde r_j u_{n\bar p} - \tilde r_j (u_n-u_{\bar n})_{\bar p} \\
&=	2 \tilde r_j g_{n \bar p} - 2 \tilde r_j \chi_{n \bar p} - \tilde r_jV_{\bar p},
\end{aligned}
\]
where we wrote $V= u_n-u_{\bar n}.$ 

By Lemma~\ref{i2-main}, we have for $\cF|\lambda|:= \sum_{i}f_i|\lambda_i|$,
\[ \notag
	|2L^{\bar p j}\tilde r_j g_{n\bar p}| \leq C|L^{\bar p j} g_{n\bar p}	| \leq C \cF|\lambda|.
\]
A straightforward estimate gives
\[ \notag
	|2L^{\bar p j} \tilde r_j \chi_{n\bar p}| \leq C\cF.
\]
Cauchy-Schwarz's inequality implies that
\[ \notag
\begin{aligned}
	|L^{\bar p j} \tilde r_jV_{\bar p}| 
&\leq \frac{1}{2}L^{\bar p j}\tilde r_j\tilde r_{\bar p} + \frac{1}{2} L^{\bar p j} (\bar V)_{j} V_{\bar p} \\
&\leq C \cF + \frac{1}{2} L^{\bar p j} (\bar V)_{j} V_{\bar p}.
\end{aligned}
\]
Thus, the above estimates give
\[ \label{est-i2}
	|I_2| \leq C(\cF + \cF|\lambda|) 
	+ L^{\bar p j} (\bar V)_{j} V_{\bar p}.
\]

{\bf (1b) Estimate $I_1$.} We have
\[ \notag
\begin{aligned}
&	u_{1j\bar p} = u_{j\bar p 1} = g_{j\bar p 1} - \chi_{j\bar p 1} .\\
\end{aligned}
\]
Covariant differentiation in direction $\partial/\partial z_1$ of the equation $F(A) =h$ gives
\[ \label{eq:diff-c2-extra}
	F^{ij} \alpha^{\bar p i} \nabla_1 g_{j\bar p} 
	= L^{\bar p j} \left[g_{j\bar p 1} - \Gamma_{1j}^q g_{q\bar p}\right]= h_1.
\]
It follows that
\[ \label{use-diff-eq}
\begin{aligned}
	|L^{\bar p j} u_{1 j\bar p}| 
&=	|L^{\bar p j}  (g_{j\bar p 1} - \chi_{j\bar p 1})| \\
&= 	|h_k + L^{\bar p j}\Gamma_{1j}^q g_{q\bar p}  - L^{\bar p j} \chi_{j\bar p 1}| \\
&\leq		C(1 + \cF)   + |L^{\bar pj}\Gamma_{1j}^q g_{q\bar p}| \\
&\leq 	C(1+ \cF + \cF|\lambda|),
\end{aligned}
\]
where we used  Lemma~\ref{i2-main} for the last  inequality.

The remaining terms in $I_1$ are estimated similarly, when the index $1$ is replaced by $\bar 1, \bar n$ or $n$. Therefore,
\[\label{est-i1}
|I_1| \leq C(1+ \cF + \cF|\lambda|).
\]
Combining \eqref{est-i3}, \eqref{est-i2} and \eqref{est-i1} yields 
\[\label{all-123}
|L^{\bar p j} (D_1u)_{j\bar p}| \leq C(1+ \cF + \cF|\lambda|) 
+  L^{\bar p j} (\bar V)_{j} V_{\bar p} .
\] 

We continue to estimate the other terms in the formula for $w$.

{\bf (2) Estimate for $-\sum_{k<n} |u_k|^2$.} By computing

\[ \label{sum-term}
\begin{aligned}
	(u_k u_{\bar k})_{j\bar p} 
&= u_{k j\bar p}u_{\bar k} + u_{k} u_{\bar k j \bar p} 
	+ u_{kj} u_{\bar k \bar p} + u_{k\bar p} u_{\bar k j}. \\
\end{aligned}
\]
Similarly to the estimation of $I_1$, we have
\[ \notag \begin{aligned}
\sum_{k<n} |L^{\bar p j} (u_{kj\bar p} u_{\bar k} + u_{k} u_{\bar k j\bar p})| 
&\leq C K^\frac{1}{2} (1+ \cF + \cF|\lambda|) \\
&\leq C(1+ \cF + \cF|\lambda|).
\end{aligned}\]
For the third term, with $k$  fixed,
$
	L^{\bar p j} u_{kj} u_{\bar k \bar p} \geq 0.
$
The last term in \eqref{sum-term} will give a good positive term. By using Lemma~\ref{i2-main},
\[ \label{po-sum-term}
\begin{aligned}
	L^{\bar p j} u_{k\bar p} u_{\bar k j} &= L^{\bar p j} (g_{k\bar p} - \chi_{k\bar p})(g_{j\bar k} -\chi_{j\bar k}) \\
&\geq L^{\bar p j} g_{k\bar p} g_{j \bar k} - C(\cF + \cF|\lambda|).
\end{aligned}
\]

The following result is similar to Guan's \cite[Proposition 2.19]{guan14} in the real case.
\begin{lem} \label{est-sum-term}
There exists an index $s$ such that
\[\notag
	\sum_{k<n} L^{\bar p j} g_{k \bar p} g_{j\bar k}
\geq 	\frac{\min_i\tau_{i}}{2} \sum_{i\neq s} f_{i} \lambda_{i}^2,
\]
where $\tau_i$'s are the eigenvalues of the matrix $\alpha_{i\bar j}$.
\end{lem}

\begin{proof}[Proof of Lemma~\ref{est-sum-term}]
First at the given point let $E=(e_{ij})$ be a unitary matrix such that $\alpha=E^t\Lambda \bar E$, where $\Lambda=diag(\tau_1,\dots, \tau_n).$ Without loss of generality, we can assume $X=\Lambda^{\frac{1}{2}}E$,  so that $\alpha=X^t\bar X$ and $x_{ij}=\tau_i^\frac{1}{2}e_{ij}$. Again we have formulas $\alpha_{i\bar j} = x_{ai} \overline{x_{aj}}$ and $\alpha^{\bar i j} = \overline{x^{ia}} x^{ja}$. Moreover,
\[\notag
	L^{\bar p j} = \overline{x^{pa}} f_a x^{ja},\quad
	g_{i\bar j} = x_{ib} \lambda_b \overline{x_{b j}}.
\]
Thus, for a fixed $k<n$,
\[\notag
\begin{aligned}
L^{\bar p j} g_{k \bar p} g_{j\bar k} 
&=	 \overline{x^{pa}} f_a x^{ja} \; x_{bk}\lambda_b \overline{x_{bp}} \;x_{cj} \lambda_c  \overline{x_{ck}} \\
	&= \sum_{i=1}^n f_i \lambda_i^2 |x_{ik}|^2. \\
\end{aligned}
\]
As 
\[\notag
\sum_{k<n} |x_{ik}|^2 =\sum_{k=1}^n|x_{ik}|^2 - |x_{in}|^2=\tau_i(1-|e_{in}|^2),
\]
 we have
\[\notag
S:= \sum_{k<n} L^{\bar p j} g_{k \bar p} g_{j\bar k}  =
	\sum_{i=1}^n f_{i} \lambda_i^2\tau_i (1-|e_{in}|^2). 
\]
 If for every $1\leq i \leq n$
we have $|e_{in}|^2 \leq \frac{1}{2}$, then
\[\notag
	S  \geq \frac{\min_i\tau_{i}}{2} \sum_{i=1}^n f_i \lambda_{i}^2.
\]
Otherwise, there exists an index $s$ such that $|e_{sn}|^2> \frac{1}{2}$. It follows that
 \[ \notag
 \sum_{i \neq s} |e_{in}|^2 \leq \frac{1}{2}.\] Then,
\[\notag
S=	\sum_{i=1}^n f_{i} \lambda_i^2\tau_i (1-|e_{in}|^2) \geq 
	\sum_{i\neq s} f_{i} \lambda_i^2\tau_i (1-|e_{in}|^2)\geq
	   \frac{\min_i\tau_{i}}{2} \sum_{i\neq s} f_{i} \lambda_{i}^2.
\]
Thus, the lemma follows.
\end{proof}

It follows from Lemma~\ref{est-sum-term} and \eqref{po-sum-term} that for some index $s$,
\[ \notag
	\sum_{k<n} L^{\bar p j} u_{k\bar p} u_{\bar k j} \geq \frac{\min_i\tau_{i}}{2} \sum_{i\neq s} f_{i} \lambda_{i}^2 - C(\cF + \cF|\lambda|).
\]
Therefore,
\[ \label{ne-sum-term-2}
L(-\sum_{k<n}|u_k|^2) \leq -  \frac{\min_i\tau_{i}}{2} \sum_{i\neq s} f_{i} \lambda_{i}^2 + C(\cF + \cF|\lambda|).
\]

{\bf (3) Estimate for  $-|V|^2 = - |u_n-u_{\bar n}|^2$.} We compute 
\[\notag
\begin{aligned}
	(V\bar V)_{j\bar p} 
&=	(u_{nj\bar p} - u_{\bar n j\bar p}) \bar V + V (u_{\bar n j\bar p} - u_{n j\bar p}) \\
&\quad	+	V_j (\bar V)_{\bar p} + (\bar V)_{j}V_{\bar p}. 
\end{aligned}
\]
Since $L^{\bar pj} V_j (\bar V)_{\bar p} \geq 0$,  we get, similarly to \eqref{use-diff-eq}, the following 
\[ \label{ad-fctn3} 
\begin{aligned}
	L^{\bar p j}(-|V|^2)_{j \bar p} 
\leq  - L^{\bar p j} (\bar V)_j V_{\bar p}+ C (1 + \cF +  \cF|\lambda|).
\end{aligned}
\end{equation}

Combining \eqref{all-123}, \eqref{ne-sum-term-2} and \eqref{ad-fctn3}  gives us
\[ \notag
Lw \leq -  \frac{\min_i\tau_{i}}{2} \sum_{i\neq s} f_{i} \lambda_{i}^2 +C(1+ \cF + \cF|\lambda|) +\mu_1 Lb + \mu_2 L(|z|^2).
\]
By this and Lemma~\ref{bfct} we get that, for some index $s$,
\[ \label{end1}
	Lw \leq - 	\frac{\min_i\tau_{i}}{2} \sum_{i\neq s} f_{i} \lambda_{i}^2 + C\cF |\lambda| + (C+ \mu_2  -\frac{\mu_1}{2}) \cF.
\]
Recall that $\mu_2$ was chosen to have the property $(ii)$ and $\mu_1>0$ can be chosen freely. To archive the third property of $w$ we need the following 

\begin{lem} \label{el-lem}
Let $\vepsilon>0$. There is a constant $C_\vepsilon>0$ such that for any index $s$,
\[ \notag
	\cF|\lambda| = \sum_{i=1}^n f_i |\lambda_i| \leq \vepsilon \sum_{i\neq s} f_i \lambda_i^2 + C_\vepsilon \cF.
\]
\end{lem}
\begin{proof}[Proof of Lemma~\ref{el-lem}] 
Since $\sum_{i=1}^n f_i \lambda_i = h,$ we have 
\[\notag
\begin{aligned}
	\cF |\lambda| &\leq  2 \sum_{i\neq s} f_i |\lambda_i|  + h \\
&\leq		\sum_{i\neq s} f_i(\vepsilon   \lambda_i^2 + \frac{1}{\vepsilon}) +h \\
&\leq 	\vepsilon \sum_{i\neq s} f_i \lambda_i^2 + C_\vepsilon \cF,
\end{aligned}
\]
where we used the fact that $\cF$ is uniformly bounded below by a positive constant.
\end{proof}

Using Lemma~\ref{el-lem} we get from  \eqref{end1} that
\[ \notag
Lw \leq (-\frac{\min_i\tau_i}{2}+ C \vepsilon) \sum_{i\neq s} f_{i} \lambda_{i}^2  + (\mu_2 + C+C_\vepsilon - \frac{\mu_1}{2}) \cF.
\]
Thus, we choose $\vepsilon$ so small that the first term on the right hand side is negative, and then choose $\mu_1$ so large that the second term is also negative. The third property $(iii)$ is proved.

We are ready to conclude the bound for tangential-normal second derivatives. By the maximum principle we have $w \geq 0$ on $U$. Therefore, as $w(0) =0$,
$
	D_{2n} w(0) \leq 0.
$
It follows that \[ \notag D_{2n}D_1u(0) \leq C.\] 

The properties $(i), (ii)$ and $(iii)$ also hold, with the same argument, if we replace $w$ by the function
\[\notag
	\tilde w = - D_1u  -\sum_{k<n} |u_k|^2 - |u_n-u_{\bar n}|^2 + \mu_1b + \mu_2 |z|^2.
\]
Therefore, 
$
	D_{2n}D_1u(0) \geq - C.
$
Thus, we get the desired bound for $|D_{2n}D_1u(0)|$. 
\end{proof}

The last estimate we need is the normal-normal derivative bound.

\begin{lem} \label{n-n-der} We have
\[ \notag\left|\frac{\d^2 u}{\d x_n^2 } (0)\right| \leq C,\]
where $C$ depends on $h,C_0, C_1,$ and the bounds of tangential-normal derivatives.
\end{lem}

\begin{proof} 
Since 
$
	4u_{n \bar n} = \d^2u/ \d x_n^2 + \d^2u/ \d y_n^2,
$
the normal-normal estimate is equivalent to   \[ \notag |u_{n \bar n}(0)| \leq C.\]
Moreover, as $|u_{i\bar j}| < C$ with $i+j < 2n$, we get that for $j<n$,
\[ \notag
	|A^{i}_j| = |\alpha^{\bar p i} (\chi_{j\bar p} + u_{j\bar p})| < C. 
\]
Hence, it follows from  \[\notag\sum_{i=1}^n A^{i}_i = \sum_{i=1}^n \lambda_i \geq 0\] 
that $A^n_n \geq - C$, so is $g_{n\bar n} \geq -C$. It implies that 
$
	u_{n\bar n} \geq - C.
$
Therefore, it remains to prove that
$
	u_{n\bar n} \leq C.
$
By 
$
	u = \sigma r,
$ with $\sigma>0$, 
we have for $j, k <n$, 
\[ \notag
	u_{j \bar k} (0) = \sigma(0) r_{j \bar k}(0).
\]
Let $S$ be a $(n-1) \times (n-1)$ unitary matrix diagonalising $[u_{j\bar k}]_{j,k<n}$. It means that for $j, k<n$,
\[ \notag
	u_{j\bar k}(0) = \sum_{p} S^*_{j p} d_p S_{pk}
\]
Since $r$ is strictly plurisubharmonic  in $U$, we get that $d_p>0$, $p=1,...,n-1$. By elementary matrix computation we have for $D=(d_1, ..,d_{n-1})$ a diagonal matrix and the column vector $V = (u_{1\bar n},..., u_{(n-1) \bar n})^t$, 
\[ \notag
\begin{pmatrix}
	S & 0 \\
	0 & 1
\end{pmatrix}
\times [u_{i\bar j}]_{i,j\leq n} \times
\begin{pmatrix}
	S^* & 0 \\
	0 & 1
\end{pmatrix}
= \begin{pmatrix}
 D & S V \\
 V^* S^* & u_{n\bar n}
\end{pmatrix}.
\]
By $|u_{j \bar n}|, |u_{n\bar j}| <C $ for $j<n$ and $\chi_{i\bar j} >0$, we may assume that $u_{n\bar n}$ is so large (otherwise we are done) that 
$
g_{i\bar j} = \chi_{i\bar j} + u_{i\bar j}(0) >0, 
$ i.e., positive definite. So 
\[ \notag
	\lambda_i(A) >0
\]
for every $i =1,..,n$. Hence,
\[ \notag
	(\det A)^\frac{1}{n} \leq C_{m,n} [S_m(\lambda(A))]^\frac{1}{m} =C_{m,n} h.
\]
By 
$
	\det g_{i\bar j} = \det \alpha_{i\bar j} \cdot \det A^{i}_j
$
we get that $\det g_{i\bar j} \leq C$. Since \[\notag [g_{i\bar j}]_{i,j<n}\geq [\chi_{i\bar j}]_{i,j<n} >0\] and
\[ \notag
	\det g_{i\bar j} = g_{n\bar n} \det ([g_{i\bar j}]_{i,j<n}) + O(1),
\]
we have $g_{n\bar n} \leq C$.
Thus, the normal-normal derivative bound at a boundary point is proven.
\end{proof}

Altogether, we have proven the $C^2-$estimate \eqref{c2-global} and completed the proof of Theorem~\ref{thm:intr-1}.

\section{Appendix} \label{sec:appendix}

The results in this section are classical. It is a natural generalisation of properties of  subharmonic functions (see e.g \cite{H94}). However, we could not find the the precise forms that we need in the literature. Some of them  have been pointed out recently by Harvey-Lawson \cite{hl13}. Our setup here is simpler than the one in \cite{hl13}, therefore we have several finer properties. We emphasize here the use of a theorem of Littman \cite{littman63}. For the readers' convenience  we give results with proofs here.

\subsection{Littman's theorem} \label{sec:littman}

We briefly recall a simpler version of a result of Littman \cite{littman59,littman63}. Roughly speaking it allows to approximate a generalised subharmonic function (with respect to a uniformly elliptic operator $L$) in a constructive way.

 Let $D$ be a smoothly bounded domain in $\bR^n$, $n\geq 3$. Consider the partial differential operator $L$ defined by
\[\notag
	L u = \big(b^{ij} u\big)_{x_ix_j} - \big( b^i(x) u\big)_{x_i}
\]
and assumed to be uniformly elliptic there. Its formal adjoint $L^*$ is given by 
\[\notag
	L^*v = b^{ij}(x) v_{x_i x_j} + b^i(x) v_{x_i},
\]
where coefficients $b^{ij}(x)$, $b^i(x)$ are smooth function on $D$.

We say that $u \in L^1_{loc}(D)$ satisfies $Lu \geq 0$ weakly if
\[\label{eq:w-lsh}
	\int u(x) L^* v(x) \geq 0
\]
for any non-negative function $v$ in $C^2(D)$ with compact support in $D$. The natural question is to find a sequence of smooth functions $u_j$ such that $L u_j \geq 0$ and $u_j$ decreases to $u$. The usual convolution with a smooth kernel will not give us the desired sequence.

Before stating Littman's theorem let us introduce some notations.
We denote by $g(x,y)$ the Green function of the operator $L_x$ with respect to domain $D$ and with singularity at $y\in D$; as constructed for example in \cite{miranda}. The subindex $x$ means that $L$ acts on functions of $x$. The basic properties of $g$ are:
\[\notag
	L^*_x g(x,y) = 0 \quad  \mbox{on } D \setminus \{y\},
\]
and 
$$ \quad g(x,y) = O\big(|x-y|^{2-n}\big) \mbox{ as } x\to y.
$$
In particular, $g(x,y)\to \infty$ as $x\to y$. Furthermore, let us denote $\Delta = \{(x,x): x \in \bar D\}$, then 
\[ \notag
g \in C^0(\bar D \times \bar D \setminus \Delta) \cap C^2(D\times D \setminus \Delta);
\]  
also $g(x,y) =0$ for $x \in \d D$ and a fixed $y\in D$. If we denote $r = |x-y|$. Then
\[\notag\begin{aligned}
	g(x,y) = O(r^{2-n}), \quad g_{x_i} = O(r^{1-n}), \quad g_{x_i x_j} = O(r^{-n}).
\end{aligned}
\]

Fix a function $p(t) = 1 -t^2$ for $t\in \bR$. So $p(0) =1$ and $p(t) \geq \delta_0>0$ for $|t| < \delta_0$ small enough. It is easy to see that
\[\notag
	L^*_x p(|x-y|) < 0 \quad \mbox{for } |x-y| < 2 \delta_0 \mbox{ and } x,y \in D.
\]
Let $\Phi(t) \geq 0$ be a smooth function supported on $[0,1]$ satisfying
\[\notag\begin{aligned}
&	\Phi(t) \to 0 \quad \mbox{exponetially as } t \to +\infty; \\
&	\int_{-\infty}^{\infty} \Phi(t) dt =1. 
\end{aligned}
\]

For $\delta>0$ we write $D_\delta = \{z\in D : dist(z, \d D) >\delta\}$. For $h\geq 0$, $x\in D$, $y\in D_{2\delta}$ we define a function $G_h(x,y)$ on $D \times D_{2\delta}$ by letting
\[\notag
	G_h(x,y) := 0  \quad \mbox{for} \quad |x-y| \geq 2\delta,
\]
and for $|x-y| < 2 \delta,$
\[\notag
	G_h(x,y) := 
	\int_{-\infty}^\infty \Phi(s -h) \max\{g(x,y)- s p(x,y),0\} ds. 
\]

Notice that $G_h(x,y) =0$ for $|x-y| \geq \delta$ and $h \geq h_\delta$, where
\[ \label{eq: h-delta}
	h_\delta := \frac{1}{\delta_0} \max\big\{g(x,y) :\delta \leq |x-y| \leq 2\delta, (x,y) \in D\times D_{2\delta}\big\}.
\]
Another remark is that 
\[\label{eq:a-levi}
	G_h(x,y) - g(x,y) = g(x,y) \int_{g/p}^\infty \Phi(s-h) ds - p \int_{-\infty}^{g/p} \Phi(s -h) ds
\]
is continuous for $x\in \bar D$ and $y \in D_{2\delta}$ and it belongs to $C^2(D\times D_{2\delta})$ as the rate of $g(x,y)$ growing to $\infty$ is polynomially while $\Phi(t) \to 0$ exponentially. In particular, $G_h(x,y) \to +\infty$ as $x \to y$ with the same order of growth as $g(x,y)$.

By a direct computation we get that
\[ \label{eq: dgreen}
	\frac{\d G_h}{\d h} = - p \int_{-\infty}^{g/p} \Phi(s -h) ds \leq 0.
\]
The formula also shows that $\frac{\d G_h}{\d h} \in C^2(D)$ and compactly supported as a function of $x$. Hence,
\[\notag
J_h:= \int L^*_x G_h dx =1
\]
for every $h \geq h_\delta$. Indeed, by the property \cite[4.f]{littman63} we have $ \lim_{h \to \infty} J_h =1$, and  for any constant $c$ we have $L c =0$. Therefore,
\[\notag
	\d J_h / \d h = \int L^*_x \big(\d G_h/\d h \big)dx= \int \d G_h/\d h \; L_x 1=0.
\]

Since coefficients $b^{ij}(x), b^i(x)$ are smooth, we have
\[\notag
 	G_h(x,y) - g(x,y) \in C^2(D_{2\delta})
\]
as a function of $y$ uniformly with respect to $x$ (c.f \cite[4.e]{littman63}). Hence,  $G_h$ is a {\em Levi function} satisfying
\[\notag
	L^*_x G_h(x,y) = O(|x-y|^{\lambda - n})
\]
for any $0 < \lambda \leq 1$ (c.f \cite[(8.5) p. 18]{miranda}). Therefore, for  $u \in L^1_{loc}(D)$ and $h\geq h_\delta$,
\[ \label{eq:u_h}
	u_h(y) = \int u(x) L^*_x G_h (x,y) dx = \int_{|x-y| \leq \delta} u(x) L^*_xG_h(x,y) dx
\]
is well defined. Notice that the support of $G_h(x,y)$, as a function in $x$, shrinks to $y$ as $h\to +\infty$.

We are ready to state a theorem of Littman \cite{littman63}. 

\begin{thm}[Littman]\label{thm:littman} Let $u \in L^1_{loc}(D)$ be such that $Lu\geq 0$ weakly in $D$ in the sense of \eqref{eq:w-lsh}. Then, $\{u_h(x)\}_{h \geq h_\delta}$, defined by \eqref{eq:u_h}, are smooth functions satisfying:
\begin{itemize}
\item
$L u_h \geq 0$;
\item
$u_h$ is a nonincreasing sequence as $h \to +\infty$, $u_h$ converges to $u$ in $L^1(D_{2\delta})$;
\item
$U(x) := \lim_{h\to \infty} u_h(x)$ is upper semicontinuous, and $U(x) = u(x)$ almost everywhere.
\end{itemize}
\end{thm}

\subsection{Properties of $\omega-$subharmonic functions} \label{sec:omega-sub}

Let $\omega$ be a Hermitian metric on a bounded open set $\Omega \subset \bC^n$. 
Let us denote
\[
	\Delta_\omega := \omega^{\bar j i}(z) \frac{\d^2}{\d z^i \d\bar z^j}.
\]
We first recall 

\begin{defn} \label{def-1-omega-sh}
A function $u: \Omega \to [-\infty,+\infty[$ is called $\omega-$subharmonic  if 
\begin{enumerate}
\item[(a)]  $u$ is upper semicontinuous and $u\in L^1_{loc}(\Omega)$. 
\item[(b)] for every relatively compact open set $D\subset\subset \Omega$ and every $h \in C^0(\bar D)$ and $\Delta_\omega h =0$ in $D$, if $h \geq u$ on $\d D$, then $h \geq u$ on $\bar D$.
\end{enumerate}
\end{defn}

As in the case of subharmonic functions we have the following properties. These properties are proved by using the above definition (see also \cite[Theorem~3.2.2]{H94}). 

\begin{prop} \label{proa:limit-stable}
Let $\Omega$ be a bounded open set in $\bC^n$. 
\begin{enumerate}
\item[(a)]
If $u_1\geq u_2 \geq\cdots$ is a decreasing sequence of $\omega-$subharmonic functions, then $u := \lim_{j\to +\infty} u_j$ is either $\omega-$subharmonic or $\equiv-\infty$.
\item[(b)]
If $u, v$ belong to $SH(\omega)$, then so does $\max\{u,v\}$.
\end{enumerate}
\end{prop}

\begin{proof} $(a)$ is obvious. We shall prove $(b)$.
It is rather standard (see \cite{GZ05}), but probably it is not so well known. We include the proof for the sake of completeness. Observe that
\[\notag
	\max\{u,v\} = \lim_{j \to +\infty} \frac{\log (e^{j u} + e^{jv})}{j} 
\]
By a simple computation we get that
\[\notag \begin{aligned}
	dd^c \log(e^u + e^v) &= \frac{e^u dd^c u+e^v dd^c v}{(e^u+e^v)} + \frac{e^{u+v}d(u-v)\wed d^c(u-v)}{(e^u+e^v)^2}.
\end{aligned}\]
It follows that $\frac{1}{j}\log (e^{ju}+ e^{jv})$ is $\omega$-subharmonic. So is $\max\{u,v\}$.
\end{proof}

The subharmonicity is a local notion meaning that a function is subharmonic in a open set if and only if at every point there exists a neighbourhood such that the function is subharmonic in that neighbourhood. The precise statement is  

\begin{prop} The following statements are equivalent for an upper semicontinuous  and locally integrable function $u$ in $\Omega$.
\begin{itemize}
\item[(1)]
$u$ is an $\omega-$subharmonic function in $\Omega$.
\item[(2)]
In a neighbourhood $U$ of a given point $a$, if $q \in C^2(U)$ such that $q - u \geq 0$ and $q(a) = u(a)$, then 
\[\notag
	\Delta_\omega q(a) \geq 0.
\] 
\end{itemize}
\end{prop}

\begin{proof} 
$(1)\Rightarrow (2).$ We argue by contradiction. Suppose that there exist a neighbourhood $U$ of a point $a$ and $q\in C^2(U)$ satisfying $q \geq u$ and $q(a) = u(a)$, but 
\[\notag
	\Delta_\omega q(a) < 0.
\]
By Taylor's formula we may assume that $q$ is  quadratic and  there exists $\vepsilon>0$ such that $\Delta_\omega q <-\vepsilon$ on a small ball $B(a, r)$. Solve
\[\notag
	\Delta_\omega v(z) =  - \Delta_\omega q,  \quad v = 0 \quad\mbox{on } \d B(a,r).
\]
Notice that by maximum principle we get that $v(a) < 0$. Let $h= v+q$. Then, $\Delta_\omega h =0$, and $h \geq u$ on $\bar B(a,r)$. However,  $h(a) = u(a) + v(a) < u(a)$, which is impossible. The first direction follows.

$(2)\Rightarrow (1).$ We also argue by contradiction. Suppose that there exist an open set $D \subset \subset \Omega$ and a function $h\in C^0(\bar D)$ and $\Delta_\omega h = 0$ in $D$, which satisfies $u \leq h$ on $\d D$, such that  $\{u>h\}$ is non-empty.  
Without loss of generality we may assume that $D$ is a small ball $B$ and $h$ is continuous on $\bar B$. Set for $\vepsilon>0$
\[\notag
	v_\vepsilon(z) = h(z) - \vepsilon |z|^2.
\]
Then, the upper semicontinuous function $(u-v_\vepsilon)$ takes its maximum at a point $a \in B$, so
\[\notag
	u(z) \leq v_\vepsilon(z) + u(a) - v_\vepsilon(a) \quad \mbox{for } z\in B.
\]
By Taylor's formula 
\[\notag\begin{aligned}
	h(z) &= h(a) + \Re (P(z)) + \frac{1}{2} \frac{\d^2 h}{\d z_i \d \bar z_j}(a) (z_i -a_i) \overline{(z_j-a_j)} + O(|z-a|^3) \\
& =: H(z) + O(|z-a|^3),
\end{aligned}\]
where $P(z)$ is a holomorphic polynomial. Therefore, $\Delta_\omega H(a) = 0$. Consider the function 
\[\notag
q(z)= u(a) - v_\vepsilon(a) +  H(z) -\vepsilon |z|^2 + \frac{\vepsilon}{2} |z-a|^2.
\]
Then, it is easy to check that $\Delta_\omega q(a) < 0$, $q(a) = u(a)$ and $q(z) \geq u(z)$ in a neighbourhood of $a$. This is impossible and the proof is completed.
\end{proof}

Since $\omega-$subharmonicity is a local property  we easily get the gluing lemma.

\begin{lem} Let $U \subset V$ be two open sets. Let $u \in SH(\omega, U)$ and $v \in SH(\omega, V)$. Assume that \[ \limsup_{z\to \zeta} u(z) \leq v(\zeta) \quad \forall \zeta \in \d U\cap V. \]
Then,  $\tilde u \in SH(\omega, V)$, where
\[\notag
\tilde u = \begin{cases} \max\{u,v\} \quad &\mbox{on } U, \\
v\quad &\mbox{on } V\setminus U. 
\end{cases}
\]
\end{lem}

\begin{proof} Consider 
\[\notag
	u_\vepsilon = \begin{cases}
	\max\{u, v+\vepsilon\} \quad \mbox{on } U, \\
	v+\vepsilon \quad \mbox{on } V\setminus U.
	\end{cases}
\]
If $x \in U$, then there is a small ball $B(x,r) \subset U$. Hence, \[
 u_\vepsilon =  \max\{u, v+\vepsilon\} \] is $\omega$-subharmonic in $B(x,r).$ Similarly, for $x\in V\setminus U$ by the assumption on $\d U \cap V$, there is $B(x,r) \subset V$ such that $u_\vepsilon = v+ \vepsilon$ on $B(x,r)$. Thus, $u_\vepsilon \in SH(\omega,V).$ Since $u_\vepsilon \searrow  u$ we can apply Proposition~\ref{proa:limit-stable} getting the lemma.
\end{proof}

The proposition above shows that we only need to check the $\omega-$subharmonicity of a function on a small ball, but it is not clear whether a sum of two subharmonic functions is again  subharmonic. We shall need another criterion. 

By linear PDEs potential theory, e.g. see \cite{miranda}, for any ball $B(a,r)$, there exists a Poisson kernel $P_{a,r}$ for the operator $\Delta_\omega$. Namely, for every continuous function $\vphi$ on $\d B(a,r)$, the function
\[\notag
	h(z) = \int_{\d B(a,r)} \vphi(w) P_{a,r}(z,w) d\sigma_r(w),
\]
is the unique continuous solution to the Dirichlet problem 
\[\notag
	\Delta_\omega h(z) = 0 \quad \mbox{in } B(a,r), \quad
	h = \vphi \mbox{ on } \d B(a,r),
\]
where $d\sigma_r(z)$ is the standard surface measure on $\d B(a,r)$.

\begin{lem} Let $u: \Omega \to [-\infty, +\infty[$ be a locally integrable upper semicontinuous function. For $\Omega_\delta = \{z\in \Omega : dist(z, \d\Omega) >\delta\}$, $\delta >0$, consider the function 
\[\notag
	M(u,a,r) =	\int_{\d B(a,r)} u(z) P_{a,r}(a,z) d\sigma_r(z), \quad a \in \Omega_\delta,
\]
where $r \in [0,\delta]$.
Then,  $u$ is an $\omega-$subharmonic function if and only if  
\[\notag
 u(a) \leq M(u,a,r)
\] for $a \in \Omega_\delta$, $r \in [0,\delta].$ Furthermore, $M(u,a,r)$ decreases to $u(a)$ as $r$ goes to $0$.
\end{lem}

\begin{proof} We first prove that it is a necessary condition. Take $\phi \geq u$ to be a continuous function on $\d B(a,r)$. Then,
\[\notag
	h(z) = \int_{\d B(a,r)} \phi(w) P_{a,r} (z,w) d\sigma_r(w)
\]
satisfies $\Delta_\omega h =0$ and $h=\phi \geq u$ on $\d B(a,r)$. It follows from definition that $h \geq u$ on $B(a,r)$. In particular,
\[\notag
	u(a) \leq \int_{\d B(a,r)} \phi(w) P_{a,r} (a,w) d\sigma_r(w).
\]
As $u$ is upper semicontinuous, we can let $\phi \searrow u$. By monotone convergence theorem we  get the desired inequality.

Now we prove the other direction by contradiction. Assume that there exist a relatively compact open set $D \subset \Omega$, $h \in C^0(\bar D)$ with $\Delta_\omega h =0$ and $h \geq u$ on $\d D$, but 
\[\notag
	c:= \sup_{\bar D} (u -h) >0.
\]
As $v= u-h$ is upper semicontinuous, $c$ is finite  and $F:=\{v = c\}$ is a compact set in $D$. We choose $a\in F$ such that it is the closest point to the boundary $\d D$. Assume that $dist (a, \d D) = 2\delta>0.$ Since there exists $x \in B(a,\delta)$ such that $v(x) < c$, so there is $B(x,\epsilon') \subset \{v<c-\epsilon\} \cap B(a,\delta)$ for some $\epsilon,\epsilon'>0$. It follows from $\Delta_\omega h =0$ on $D$ that
\[\notag
	v(a) \leq M(v, a,r) \quad \forall z\in B(a,r), \forall r \leq \delta. 
\]
Notice that in our case $\Delta_\omega 1 =0$ and
\[\notag
	\int_{\d B(a,r)} P_{a,r} (z,w) d\sigma_r(w) =1.
\]
Integrating from $0$ to $\delta$ we get that
\[\notag
	\delta v(a) \leq \int_{[0,\delta]} \int_{\d B(a, r)} v(z) P_r(a,z) d\sigma_r(z) dr < \delta c.
\]
This is impossible. Thus, the sufficient condition is proved.

For the last assertion, let $0 \leq r<\delta$.  Fix a continuous function $\phi \geq u$ on $\d B(0,\delta)$. As $\Delta_\omega h =0$ in $B(a,\delta)$ for
\[\notag
	h(z) = \int_{\d B(a, \delta)} \phi(w) P_\delta(z, w) d\sigma_\delta(w),
\]
we get that $u(z) \leq h(z)$ on $B(a,\delta)$. Therefore,
\[
M(u,a,r) \leq \int_{\d B(a,r)} h(w) P_r(a, w) d\sigma_r(w) = h(a).
\]
Moreover,
\[\notag
	h(a) = \int_{\d B(a,\delta)} \phi(w) P_\delta(a, w) d\sigma_\delta(w).
\]
Letting $\phi \searrow u$, we get the monotonicity of $M(u,a,r)$ in $r \in [0,\delta].$
Moreover, as $u$ is upper semicontinuous, 
\[\notag
	\lim_{r\to 0} M(u,a,r) \leq u^*(a)= u(a),
\]
where we used the fact above that $\int _{\d B(a,r)}P_r d\sigma_r =1.$ 
\end{proof}

An immediate consequence of the last assertion in the above lemma is 
\begin{cor} \label{cor:almost-equal} If two $\omega-$subharmonic functions are equal almost everywhere, then they are equal everywhere.
\end{cor}

We are ready to state a consequence of Littman's theorem, which says that we can always find a smooth approximation for $\omega-$subharmonic functions.
\begin{cor} \label{cor:littman}
Let $u\in SH(\omega, \Omega)$ and $\Omega' \subset \subset \Omega$.There exists a  sequence of smooth $\omega-$subharmonic functions  $[u]_\vepsilon$ decreasing to $u$ as $ \vepsilon \to 0$ on $\Omega'.$
\end{cor}

\begin{proof} We simply choose a smooth domain $D \subset \Omega$ and $\delta>0$ small such that $\Omega' \subset D_{2\delta}$ and let 
\[\label{eq:u_h-epsilon}
	[u]_\vepsilon(z) := u_{h}(z)
\]
where $u_h(z)$, $h = 1/\vepsilon > h_{\delta}$, is defined in Theorem~\ref{thm:littman}. As $U(z):=\lim_{\vepsilon} [u]_\vepsilon$ is equal to $u(z)$ almost everywhere and $u$ is also $\omega-$subharmonic, it follows from Corollary~\ref{cor:almost-equal} that $U=u$ everywhere.
\end{proof}

\begin{cor} \label{cor:sup-family}
Let $\{u_\alpha\}_{\alpha \in I} \subset SH(\omega)$ be a family that is locally bounded from above. Let $u(z) := \sup_\alpha u_\alpha(z)$. Then, the upper semicontinuous regularisation $u^*$ is $\omega-$subharmonic.
\end{cor}

\begin{proof} By Choquet's lemma one can choose an increasing sequence $u_j \in SH(\omega)$ such that $u = \lim_{j} u_j$. Then, by Littman's theorem and the notation in Corollary~\ref{cor:littman}, $\lim_\vepsilon [u]_\vepsilon =U \in SH(\omega)$ and  $u = U$ almost everywhere. As $u_j \in SH(\omega)$ we have
\[\notag
	u_j \leq [u_j]_\vepsilon \to [u]_\vepsilon \quad \mbox{as } j\to +\infty
\]
uniformly on compact subsets of $\Omega$. It follows that $u \leq U$.  By upper semicontinuous of $U$ we have $u^* \leq U$. By the formula \eqref{eq:u_h} and $J_h =1$, $\lim_\vepsilon [u]_\vepsilon \leq u^*$. Thus, $u^*=U$.
\end{proof}

\begin{lem} \label{lem:positive-dis-omega-sh}
Let $u$ be an $\omega-$subharmonic function in $\Omega$. Then,
\[\notag
	\Delta_\omega u \geq 0
\]
in the distributional sense. Conversely, if $v\in  L^1_{loc}(\Omega)$ and  $\Delta_\omega v \geq 0$ (as a distribution), then there exists a unique function $V\in SH(\omega)$ such that $V = v$ in $L^1_{loc}(\Omega)$.
\end{lem}

\begin{proof} Let $[u]_\vepsilon$, $\vepsilon>0$, be the smooth decreasing approximation of $u$. As $\Delta_\omega [u]_\vepsilon \geq 0$ and the family weakly converges to $\Delta_\omega u$, we get the first statement. Conversely, by Littman's theorem we know that $V(z) = \lim_{\vepsilon \to 0} [v]_\vepsilon(z) \in SH(\omega)$ and $V(z)= v(z)$ almost everywhere. Therefore, we get the existence. The uniqueness follows from the fact that two $\omega-$subharmonic functions are equal almost everywhere.
\end{proof}

The following result is rather simple but it is important.
\begin{lem} \label{lem:l1-cpt-sh}
Let $u \in SH(\omega)$. Let $K\subset\subset D\subset\subset \Omega$ be a compact set and an open set. Then,
\[\notag
	\int_K dd^c u \wed \omega^{n-1} \leq C(D,\Omega) \|u\|_{L^1(D)}.
\]
\end{lem}

\begin{proof} Let $\phi$ be a cut-off function of $K$ and $supp \phi \subset D$. Then,
\[\notag\begin{aligned}
	\int_K dd^c u\wed \omega^{n-1} 
&\leq \int \phi dd^c u \wed \omega^{n-1} \\
&= \int u dd^c (\phi \omega^{n-1}) \\
&\leq C(D, \Omega) \|u\|_{L^1(D)},
\end{aligned}
\]
where we used that $\phi$ is smooth and has compact support in $\Omega$.
\end{proof}

\begin{lem} The convex cone $SH(\omega)$ is closed in $L^1_{loc}(\Omega),$ and it has a property that every bounded subset is relatively compact.
\end{lem}

\begin{proof} Let $u_j$ be a sequence in $SH(\omega)$. If $u_j \to u$ in $L^1_{loc}(\Omega)$, then $\Delta_\omega u_j \to \Delta_\omega u$ in weak topology of distributions, hence $\Delta_\omega u \geq 0$, and $u$ can be represented by an $\omega-$subharmonic function thanks to Lemma~\ref{lem:positive-dis-omega-sh}. 

Now suppose that $\|u_j\|_{L^1(K)}$ is uniformly bounded for every compact subset $K$ of $\Omega$. Let $\mu_j = \Delta_\omega u_j \geq 0$. Let  $\psi$ be a test function such that $0\leq \psi \leq 1$ and $\psi =1$ on $K$. Then, by Lemma~\ref{lem:l1-cpt-sh}
\[\notag
	\mu_j (K) \leq \int_\Omega \psi \Delta_\omega u_j \leq C  \|u_j\|_{L^1(K')},
\]
where $K' = \mbox{Supp } \psi$. By weak compactness $\mu_j$ weakly converges to a positive measure $\mu$. Let $G(x, y)$ be the Green kernel for the smooth domain $D$, where $K' \subset D \subset \Omega$. Consider 
\[\notag
	h_j:=u_j(z) - \int G(z,w) \psi \mu_j(w).
\]
Notice that since $G(x,y) \in L^1(d\lambda(z))$ and $\psi$ has compact support in $D$,
\[\notag
	\int G(z,w) \psi \mu_j(w) \to \int G(z,w) \psi \mu(w)
\]
in $L^1$ as $j$ goes to $+\infty$. Therefore, $\Delta_\omega h_j =0$ in $K$ and $\|h_j\|_{L^1} \leq C$. Since
\[\notag
	h_j(z) = \int_{\d D} h_j(w) P(z,w) d\sigma(w),
\]
it follows that $\|h_j\|_{C^1} \leq C$. Then, there exists a subsequence $h_j$ converging to $h$ uniformly.  Therefore,
\[\notag
	h_j + \int G(z,w) \psi \mu_j(w) \to u = h + \int G(z,w) \psi \mu(w)
\]
in $L^1(K)$ as $j$ goes to $\infty$.
\end{proof}

\begin{lem} \label{lem:hartogs-b}
Let $u_j$ be a sequence of $\omega-$subharmonic functions which are uniformly bounded above.  If $u$ is an $\omega-$subharmonic function and $u_j \to u$ in $\cD'(\Omega)$, then $u_j \to u$ in $L^1_{loc}(\Omega),$ and 
\[\notag
	\overline\lim_{j \to \infty} u_j (z) \leq u(z), \quad z\in \Omega,
\]
(where two sides are equal and finite almost everywhere).
\end{lem}

\begin{proof} By Corollary~\ref{cor:littman} for $\vepsilon>0$ small enough,
\[\label{eq:appendix-13-1}
	u_j \leq [u_j]_\vepsilon \to [u]_\vepsilon
\]
uniformly on compact sets in $\Omega$ as $j\to \infty$.   If $0\leq \phi \in C^\infty_c$ then
\[\notag
	\int ([u]_\vepsilon + \delta - u_j) \phi d\lambda(z) \to \int ([u]_\vepsilon + \delta - u) \phi d\lambda(z)
\]
as $j\to \infty$ and if $\delta >0$ the integrand is positive for $j$ large. Hence,
\[\notag
	\overline\lim_{j\to \infty} \int |u - u_j| \phi d\lambda(z) \leq 2 \int |[u]_\vepsilon + \delta - u| \phi d\lambda(z).
\]
Since $\vepsilon, \delta$ are arbitrary it follows that $u_j\to u$ in $L^1_{loc}$. 

By \eqref{eq:appendix-13-1} it is easy to see that 
$
 	\overline\lim_{j\to \infty} u_j \leq u
$ 
in $\Omega$.
Furthermore, Fatou's lemma gives
\[\notag
	\int \overline\lim \,u_j \phi d\lambda \geq \overline\lim \int u_j \phi d\lambda = \int u \phi d\lambda,
\]
so we conclude that $\overline\lim_j \,u_j = u$ almost everywhere.
\end{proof}

\begin{lem}[Hartogs] \label{lem:hartogs} Let $f$ be a continuous function on $\Omega$ and $K \subset\subset \Omega$ be a compact set. Suppose that $\{v_j\}_{j\geq 1} \subset SH(\omega)$ decrease point-wise to $v\in SH(\omega)$. Then, for any $\delta>0$, there exists $j_\delta$ such that
\[ \notag
	\sup_K (v_j -f) \leq \sup_K (v-f) + \delta 
\]
for $j \geq j_\delta$.
\end{lem}

\begin{proof} Let $[v_j]_\vepsilon$ and $[v]_\vepsilon$ be decreasing approximations defined in Corollary~\ref{cor:littman} for $v_j$ and $v$, respectively. As $v_j$ converges to $v$ in $L^1_{loc}(\Omega)$, for any fixed $\vepsilon>0$, 
\[\label{eq:uni-con-h}
	[v_j]_\vepsilon \to [v]_\vepsilon
\]
uniformly on compact sets of $\Omega$ as $j$ goes to $+\infty$. Since $v_j\leq [v_j]_\vepsilon$, we have 
\[\notag
	\sup_K(v_j -f) \leq \sup_K ([v_j]_\vepsilon - f)
\]
Let $M:= \sup_K (v-f)$. By Dini's theorem $\max\{ M, [v]_\vepsilon(z) -f(z)\}$ decreases uniformly to $M$ on $\Omega$ as $\vepsilon$ goes to $0$. Hence, for $\vepsilon>0$ small enough,
\[\notag
	\sup_K([v]_\vepsilon - f) \leq M + \delta/2.
\]
Let us fix such a small $\vepsilon$. By uniform convergence \eqref{eq:uni-con-h}, for $j \geq j_1$
\[\notag
	\sup_K ([v_j]_\vepsilon -f) \leq \sup_K([v]_\vepsilon -f) + \delta/2.
\]
Thus, altogether we get the desired inequality.
\end{proof}

A direct consequence of this lemma is 
\begin{cor} \label{cor:hartogs2}Let $\gamma$ be a real $(1,1)-$form in $\Omega$. Let $v\in SH_{\gamma, 1}(\omega) \cap L^\infty(\Omega)$. Let $\{v_j\}_{j\geq 1} \subset SH_{\gamma, 1}(\omega) \cap L^\infty(\Omega)$ be such that
\[ \notag
\lim_{j \to +\infty} v_j(z) = v(z) \quad \forall z\in \Omega.
\]
Let $K \subset \Omega$ be a compact set and $\delta>0$. Then,  there exists $j_\delta$ such that for $j\geq j_\delta$,
\[ \notag
	v_j (z) \leq \sup_K v + \delta.
\]
\end{cor}

\begin{proof}  We can find a smooth function $w$ in $\Omega$ such that
\[ \notag
	dd^c w \wed \omega^{n-1} = \gamma\wed \omega^{n-1}.
\]
As  $u_j = v_j+ w$ and $u = v+ w$ satisfy assumptions of  Lemma~\ref{lem:hartogs}, we can apply it for $f = w$ to get the statement of the corollary.
\end{proof}

\begin{cor}\label{cora:limsup-closed}
Let $\{u_j\}_{j\geq 1} \subset SH(\omega)$ be a sequence that is locally uniformly bounded above. Define $u(z) = \limsup_{j\to +\infty} u_j(z)$. Then, the upper semicontinuous regularisation $u^*$ is either $\omega-$subharmonic or $\equiv -\infty$.
\end{cor}

\begin{proof} Let $v_k = \sup_{j\geq k} u_j$. Thanks to Corollary~\ref{cor:sup-family}, $v_k^* \in SH(\omega)$ and $v_k^*$ decreases to $v \in SH(\omega)$ or $\equiv -\infty.$ Clearly, $v\geq u,$ and thus $v\geq u^*\geq u$. Since $v_k = v_k^*$ almost everywhere, so $v=u$ almost everywhere. Furthermore, it is easy to see that $\Delta_\omega u\geq 0$. By Lemma~\ref{lem:hartogs-b}
\[
	v = \lim_\vepsilon [v]_\vepsilon \leq \limsup_\vepsilon [u]_\vepsilon \leq u^*.
\]
Therefore, $v = u^*$ everywhere.
\end{proof}

We now prove  that our definition is indeed equivalent to the  definition given by Lu-Nguyen \cite[Definition 2.3]{chinh-dong}, (see also Dinew-Lu \cite{chinh-dinew}).

\begin{lem} \label{def-2-omega-sh}
A function $u:\Omega \to [-\infty, + \infty[$ is $\omega-$subharmonic if and only if it satisfies the following two conditions:
\begin{itemize}
\item[(i)] upper semicontinuous, locally integrable and $\Delta_\omega u \geq 0$ in $\Omega$.
\item[(ii)] if $v$ satisfies the condition $(i)$ and $v\geq u$ almost everywhere, then $v\geq u$ everywhere.
\end{itemize} 
\end{lem}

\begin{proof} We first show that it is a necessary conditions. The only thing that remains to be checked is the condition $(ii)$. Pick $v$ satisfying $(i)$ and $v\geq u$ almost everywehre, we wish to show that $v\geq u$ everywhere. As 
$J_h =1$, it follows from  the formulas \eqref{eq:u_h}, \eqref{eq:u_h-epsilon}, and the upper semicontinuity of $v$  that
\[\notag
	\lim_{\vepsilon\to 0} [v]_\vepsilon (z) \leq v(z).
\]
Since $[v]_\vepsilon \geq [u]_\vepsilon$ for $\vepsilon >0$, letting $\vepsilon \to 0$, we get that $v \geq u$ everywhere.

Suppose that $u$ satisfies $(i)$ and $(ii)$ above.  By Littman's theorem
$U(z) = \lim_\vepsilon [u]_\vepsilon = u(z)$ almost everywhere,
where $U(z)$ is an $\omega-$subharmonic function, which also satisfies $(i)$. Hence, $u(z) \leq U(z)$ everywhere in $\Omega.$ Moreover, using the upper semicontinuity of $u$ as above, we have $ u(z) \geq U(z)$ in  $\Omega.$
\end{proof}

We define the capacity for Borel sets $E\subset \Omega$,

\[\notag
	{\bf c}_1(E) = \sup\left\{\int_E dd^c v \wed \omega^{n-1}: 
	0\leq v \leq 1, 
	v\in SH(\omega)
\right\}.
\]
According to Lemma~\ref{lem:l1-cpt-sh} ${\bf c}_1(E)$ is finite as long as $E$ is relatively compact in $\Omega$.

The quasi-continuity of $\omega-$subharmonic functions was used in \cite{KN3}. We give here the details of the proof. First, the decreasing convergence implies the convergence in capacity.

\begin{lem} \label{con-cap-decrea}
Suppose that $u_j \in SH(\omega) \cap L^\infty(\Omega)$ and $u_j \searrow u \in SH(\omega) \cap L^\infty(\Omega).$ Then, for any compact $K\subset \Omega$ and $\delta >0$,
\[\notag
	\lim_{j\to +\infty} {\bf c}_1 (\{u_j > u+ \delta\} \cap K) =0.
\]
\end{lem}

\begin{proof} Applying the localisation principle \cite[p. 7]{kol05}, we assume that $\Omega$ is a ball and $u_j = u = h$ outside a neighbourhood of $K$. Let $0\leq v \leq 1$ be $\omega-$ subharmonic in $\Omega$. We have
\[\notag
	\int_{\{u+\delta < u_j\}\cap K} dd^cv \wed \omega^{n-1} \leq \frac{1}{\delta} \int (u_j- u) dd^c v \wed \omega^{n-1}.
\]
By Stokes's theorem,
\[\label{eq:a-decrease-cap}\begin{aligned}
\int (u_j - u) dd^c v\wed \omega^{n-1} &=  -\int d(u_j-u) \wed d^c v\wed \omega^{n-1} \\
&\quad + \int(u_j -u) d^c v \wed d \omega^{n-1}.
\end{aligned}\]

We shall show that both integrals on the right hand side tend to $0$ as $j$ goes to $+\infty$. Hence, we get the lemma.
The second one is easier. Indeed, by Schwarz's inequality \cite{cuong15}, 
\[\notag\begin{aligned}
	\left|\int (u_j -u) d^c v \wed d\omega^{n-1}\right| &\leq C \left(\int(u_j -u) dv \wed d^c v \wed \omega^{n-1}\right)^\frac{1}{2} \times \\
	&\quad \times\left(\int (u_j -u) \omega^n\right)^\frac{1}{2}.
\end{aligned}\]
Therefore the second integral of the right hand side in \eqref{eq:a-decrease-cap} goes to $0$ as $j\to +\infty$.

Similarly, we use the Schwarz inequality for the first integral in \eqref{eq:a-decrease-cap}. Let $K \subset D \subset \subset \Omega$ such that $u_j = u$ on $\Omega\setminus D$.
\[\notag\begin{aligned}
\left|\int d (u_j -u) \wed d^c v \wed \omega^{n-1}\right| &\leq C \left(\int d (u_j -u) \wed d^c(u_j-u) \wed \omega^{n-1}\right)^\frac{1}{2} \times \\
	&\quad \times\left(\int_D dv\wed d^c v  \wed \omega^{n-1}\right)^\frac{1}{2}.
\end{aligned}\]
Again by Stokes's theorem
\[\notag \begin{aligned}
&	\int d (u_j -u) \wed d^c(u_j-u) \wed \omega^{n-1} \\
&= - \int (u_j -u) dd^c(u_j-u) \wed \omega^{n-1} + \int (u_j -u) d^c(u_j-u) \wed d\omega^{n-1} \\
&= \int (u_j -u) dd^c u \wed \omega^{n-1} - \int (u_j-u) dd^c u_j \wed \omega^{n-1} \\
&\quad + \frac{1}{2}\int d^c (u_j-u)^2 \wed d \omega^{n-1} \\
&\leq \int (u_j -u) dd^c u \wed \omega^{n-1} + \frac{1}{2}\int d^c (u_j-u)^2 \wed d \omega^{n-1}.
\end{aligned}
\]
Thus, the fist integral goes to $0$ as $j\to +\infty$ by the Lebesgue dominated convergence theorem. For the second integral we use Stokes' theorem once more
\[\notag \begin{aligned}
	\int d^c (u_j-u)^2 \wed d \omega^{n-1} 
&= - \int  (u_j-u)^2 dd^c\omega^{n-1} \\
&\leq C \int (u_j - u)^2 \omega^n.
\end{aligned}
\]
The right hand side also goes to $0$ as $j\to \infty$. Thus, we get the lemma.
\end{proof}

\begin{lem} \label{lema:quasi-con} Let $u\in SH(\omega) \cap L^\infty(\Omega)$. Then for each $\vepsilon>0,$ there is an open subset $\cO$ of $\Omega$ such that ${\bf c}_1(\cO, \Omega)<\vepsilon$ and $u$ is continuous on $\Omega\setminus \cO$.
\end{lem}

\begin{proof} We may assume that $\Omega$ is a small ball because of the properties of capacity: 
\begin{itemize}
\item
if $E \subset \Omega_1\subset \Omega_2$, then ${\bf c}_1(E, \Omega_2) \leq {\bf c}_1(E,\Omega_1)$. 
\item
${\bf c}_1 (\bigcup_j E_j) \leq \sum_j {\bf c}_1(E_j)$.
\end{itemize}
Let $SH(\omega) \cap C^\infty(\Omega) \ni u_j \searrow u$ and fix a compact set $K\subset \Omega$. By Lemma~\ref{con-cap-decrea} there exists an integer $j_k$ and an open set 
\[	
	\cO_l = \{u_{j_l} >u + \frac{1}{l}\} \subset \Omega,
\]
such that ${\bf c}_1(\cO_k \cap K,\Omega)<2^{-k}$. If $G_k = \cup_{l>k} \cO_l$. Then, $u_{j_k}$ decreases uniformly to $u$ on $K\setminus G_k$. Hence, $u$ is continuous on $K\setminus G_k$.

Applying the argument above for a sequence of compact sets $K_j$ increasing to  $\Omega$ we get  open sets $G_j$ that ${\bf c}_1(G_j,\Omega)<\vepsilon 2^{-j}$.  Let $\cO = \cup_j G_j$, the lemma follows.
\end{proof}

\bigskip




\begin{thebibliography}{000}


\bibitem{BT76} E. Bedford and B. A. Taylor, 
{\it The Dirichlet problem for a complex Monge-Amp\`ere operator.} Invent. math. {\bf37} (1976), 1-44.

\bibitem{BT82} E. Bedford and B. A. Taylor, {\it A new capacity for plurisubharmonic functions.} Acta Math. {\bf149} (1982), 1--40.

\bibitem{BT87} E. Bedford\ and\ B. A. Taylor, {\it Fine topology, \v Silov boundary, and $(dd\sp c)\sp n$,} J. Funct. Anal. {\bf 72} (1987), no.~2, 225--251.

\bibitem{berman} R. Berman, {\it From Monge-Amp\`ere equations to envelopes and geodesic rays in the zero temperature limit.} preprint, arXiv:1307.3008.






\bibitem{blocki05} Z. B\l ocki, {\it Weak solutions to the complex Hessian equation,} Ann. Inst. Fourier (Grenoble) {\bf 55} (2005), no.~5, 1735--1756.

\bibitem{blocki09} Z. B\l ocki, {\it A gradient estimate in the Calabi-Yau theorem,} Math. Ann. {\bf 344} (2009), no. 2, 317--327

\bibitem{blocki12} Z. B\l ocki, {\it On geodesics in the space of K\"ahler metrics}, in  Advances in geometric analysis, 3--19, Adv. Lect. Math. (ALM), {\bf 21}, Int. Press, Somerville, MA. 


\bibitem{boucksom} S. Boucksom, {\it Monge-Amp\`ere equations on complex manifolds with boundary.}  257-282, 
Lecture Notes in Math., {\bf 2038}, Springer, Heidelberg, 2012. 




\bibitem{bru10} M. Brunella, {\it Locally conformally K\"ahler metrics on certain non-K\"ahlerian surfaces,} Math. Ann. {\bf 346} (2010), no.~3, 629--639.






\bibitem{chen00} X. Chen, {\it The space of K\"ahler metrics}, J. Differential Geom. {\bf 56} (2000), no.~2, 189--234.

\bibitem{ck94} U. Cegrell\ and\ S. Ko\l odziej, {\it The Dirichlet problem for the complex Monge-Amp\`ere operator: Perron classes and rotation-invariant measures,} Michigan Math. J. {\bf 41} (1994), no.~3, 563--569. 




\bibitem{CKNS82} L. Caffarelli, J. Kohn, L. Nirenberg  and J. Spruck, {\it The Dirichlet problem for nonlinear second-order elliptic equations. II. Complex Monge-Amp\`ere, and uniformly elliptic, equations,} Comm. Pure Appl. Math. {\bf 38} (1985), no.~2, 209--252

\bibitem{CNS85} L. Caffarelli, L. Nirenberg\ and\ J. Spruck, {\it The Dirichlet problem for nonlinear second-order elliptic equations. III. Functions of the eigenvalues of the Hessian,} Acta Math. {\bf 155} (1985), no.~3-4, 261--301.

\bibitem{chab} M. Charabati, {\it Modulus of continuity of solutions to complex Hessian equations,} Internat. J. Math. {\bf 27} (2016), no.~1, 1650003, 24 pp.


\bibitem{cherrier-hanani98} P. Cherrier, A. Hanani, {\it Le probl\`eme de Dirichlet pour des \'equations de Monge-Amp\`ere complexes modifi\'ees.} J. Funct. Anal. {\bf156} (1998), 208--251.

\bibitem{cherrier99} P. Cherrier, A. Hanani, {\it Le probl\`eme de Dirichlet pour des \'equations de Monge-Amp\`ere  en m\'etrique hermitienne.} Bull. Sci. Math. {\bf123} (1999), 577--597.





















\bibitem{DK12}S. Dinew and S. Ko\l odziej, 
{\it  Pluripotential estimates on compact Hermitian manifolds.}   Adv. Lect. Math. (ALM), {\bf 21} (2012), International Press, Boston. 

\bibitem{dk14} S. Dinew\ and\ S. Ko\l odziej, {\it A priori estimates for complex Hessian equations,} Anal. PDE {\bf 7} (2014), no.~1, 227--244. 


\bibitem{dk15} S. Dinew and S. Ko\l odziej, {\it Liouville and Calabi-Yau type theorems for complex Hessian equations.} preprint, arXiv:1203.3995. to appear in Amer. J. Math.


\bibitem{chinh-dinew} S. Dinew\ and\ C. H. Lu, {\it Mixed Hessian inequalities and uniqueness in the class $\cE (X,\omega,m)$,} Math. Z. {\bf 279} (2015), no.~3-4, 753--766.




\bibitem{don99}
S. K. Donaldson, {\it Symmetric spaces, K\"ahler geometry and Hamiltonian dynamics,} in Northern California Symplectic Geometry Seminar, 13--33, Amer. Math. Soc. Transl. Ser. 2, 196, Amer. Math. Soc., Providence.


\bibitem{EGZ11} P. Eyssidieux, V. Guedj and A. Zeriahi, {\it Viscosity solutions to degenerate complex Monge-Amp\`ere equations.} Comm. Pure Appl. Math. {\bf 64} (2011), no. 8, 1059--1094.


\bibitem{egz13} 
P. Eyssidieux, V. Guedj\ and\ A. Zeriahi, {\it Continuous approximation of quasiplurisubharmonic functions,} in Analysis, complex geometry, and mathematical physics: in honor of Duong H. Phong, 67--78, Contemp. Math., 644, Amer. Math. Soc., Providence, RI.


\bibitem{FY08} J.-X. Fu and S.-T. Yau, {\it The theory of superstring with flux on non-K\"ahler manifolds and  the complex Monge-Amp\`ere equation.} J. Diff. Geom. {\bf98} (2008), 369--428.



\bibitem{garding59} L. G\aa rding, {\it An inequality for hyperbolic polynomials,} J. Math. Mech. {\bf 8} (1959), 957--965.





\bibitem{guan98} B. Guan, {\it The Dirichlet problem for complex Monge-Amp\`ere equations and regularity of the pluri-complex Green function,} Comm. Anal. Geom. {\bf 6} (1998), no.~4, 687--703.

\bibitem{guan14} B. Guan, {\it Second-order estimates and regularity for fully nonlinear elliptic equations on Riemannian manifolds,} Duke Math. J. {\bf 163} (2014), no.~8, 1491--1524.



\bibitem{GL10} B. Guan and Q. Li, {\it Complex Monge-Amp\`ere equations and totally real submanifolds.} Adv. Math. {\bf 225} (2010), no. 3, 1185--1223.

\bibitem{GS13} B. Guan and W. Sun, {\it On a class of fully nonlinear elliptic equations on Hermitian manifolds,} Calc. Var. Partial Differential Equations {\bf 54} (2015), no.~1, 901--916.

\bibitem{GZ05} V.Guedj and A.Zeriahi, {\it Intrinsic capacities on compact K\"ahler manifolds.} J. Geom. Anal. {\bf 15} (2005), 607-639.





\bibitem{hl13}
F. R. Harvey\ and\ H. B. Lawson, Jr., {\it The equivalence of viscosity and distributional subsolutions for convex subequations---a strong Bellman principle,} Bull. Braz. Math. Soc. (N.S.) {\bf 44} (2013), no.~4, 621--652.

\bibitem{HLa}
F. R. Harvey\ and\ H. B. Lawson, Jr., {\it Dirichlet duality and the nonlinear Dirichlet problem on Riemannian manifolds,} J. Differential Geom. {\bf 88} (2011), no.~3, 395--482.

\bibitem{HLb} F. R. Harvey\ and\ H. B. Lawson, Jr., {\it Existence, uniqueness and removable singularities for nonlinear partial differential equations in geometry,} in  Surveys in differential geometry. Geometry and topology, 103--156, Surv. Differ. Geom., 18, Int. Press, Somerville, MA.

\bibitem{HLP} F. R. Harvey, \ H. B. Lawson  Jr., and S. Pli\'s,  {\it Smooth approximation of plurisubharmonic functions on almost complex manifolds,} Math. Ann. {\bf 366} (2016), no.~3-4, 929--940. 



\bibitem{H94} L. H\"ormander, Notions of convexity, Progress in Mathematics, 127, Birkh\"auser Boston, 1994

\bibitem{hou-ma-wu} Z. Hou, X.-N. Ma\ and\ D. Wu, A second order estimate for complex Hessian equations on a compact K\"ahler manifold, Math. Res. Lett. {\bf 17} (2010), no.~3, 547--561. 






\bibitem{kol96} S. Ko\l odziej, {\it Some sufficient conditions for solvability of the Dirichlet problem for the complex Monge-Amp\`ere operator}, Ann. Polon. Math. {\bf 65} (1996), 11--21.

\bibitem{kol98} S. Ko\l odziej, {\it The complex Monge-Amp\`ere equation.} Acta Math. {\bf 180} (1998), 69--117.

\bibitem{kol03} S. Ko\l odziej, 
{\it The Monge-Amp\`ere equation on compact K\"ahler manifolds.} Indiana Univ. Math. J. {\bf52} (2003), no. 3, 667--686.

\bibitem{kol05} S. Ko\l odziej, 
{\it The complex Monge-Amp\`ere equation and pluripotential theory.} Memoirs Amer. Math. Soc. {\bf 178} (2005), pp. 64.


\bibitem{KN1} S. Ko\l odziej and N.-C. Nguyen, {\it Weak solutions to the complex Monge-Amp\`ere equation on Hermitian manifolds.} Analysis, Complex Geometry, and Mathematical Physics: In Honor of Duong H. Phong, May 7-11, 2013 Columbia University, New York. Contemp. Math {\bf 644} (2015), 141-158.


\bibitem{KN2} S. Ko\l odziej and N.-C. Nguyen, {\it Stability and regularity of solutions of the Monge-Amp\`ere equation on Hermitian manifold.} preprint, arXiv: 1501.05749. 


\bibitem{KN3} S. Ko\l odziej and N.-C. Nguyen, {\it Weak solutions of complex Hessian equations on compact Hermitian manifolds,} Compos. Math. {\bf 152} (2016), no.~11, 2221--2248.


\bibitem{krylov} N. V. Krylov, {\it Boundedly inhomogeneous elliptic and parabolic equations in a domain.} (Russian) Izv. Akad. Nauk SSSR Ser. Mat. 47 (1983), no. 1, 75--108.



\bibitem{Li04}
S.-Y. Li, {\it On the Dirichlet problems for symmetric function equations of the eigenvalues of the complex Hessian,} Asian J. Math. {\bf 8} (2004), no.~1, 87--106.

\bibitem{littman59} W. Littman, {\it A strong maximum principle for weakly $L$-subharmonic functions,} J. Math. Mech. {\bf 8} (1959), 761--770. 

\bibitem{littman63} W. Littman, {\it Generalized subharmonic functions: Monotonic approximations and an improved maximum principle,} Ann. Scuola Norm. Sup. Pisa (3) {\bf 17} (1963), 207--222.




\bibitem{chinh13a} H.-C. Lu, {\it Viscosity solutions to complex Hessian equations,} J. Funct. Anal. {\bf 264} (2013), no.~6, 1355--1379.

\bibitem{chinh13b} H.-C. Lu, {\it Solutions to degenerate complex Hessian equations,} J. Math. Pures Appl. (9) {\bf 100} (2013), no.~6, 785--805. 

\bibitem{chinh15} C. H. Lu, {\it A variational approach to complex Hessian equations in $\Bbb{C}\sp n$,} J. Math. Anal. Appl. {\bf 431} (2015), no.~1, 228--259. 

\bibitem{chinh-dong} H-C. Lu and V-D. Nguyen, {\it Degenerate complex Hessian equations on compact K\"ahler manifolds,}  Indiana Univ. Math. J. {\bf 64} (2015), no.~6, 1721--1745.

\bibitem{michelsohn82} M. L. Michelsohn, {\it On the existence of special metrics in complex geometry,} Acta Math. {\bf 149} (1982), no.~3-4, 261--295.


\bibitem{miranda} C. Miranda, {\it Partial differential equations of elliptic type}, second revised edition. Translated from the Italian by Zane C. Motteler, Ergebnisse der Mathematik und ihrer Grenzgebiete, Band 2, Springer, New York, 1970.

\bibitem{cuong-thesis} N.-C. Nguyen, {\it  Weak solutions to the complex Hessian equation.} Doctoral thesis (2014), Jagiellonian University (available on researchgate.net).

\bibitem{cuong15} N. C. Nguyen, {\it The complex Monge-Amp\`ere type equation on compact Hermitian manifolds and applications,} Adv. Math. {\bf 286} (2016), 240--285.



\bibitem{PPZ} D. H. Phong, S. Picard and X. Zhang, {\it On estimates for the Fu-Yau generalization of a Strominger system,} preprint.  arXiv:1507.08193. To appear in J. Reine Angew. Math.

\bibitem{PPZ15b} D. H. Phong, S. Picard and X. Zhang, {\it 
A second order estimate for general complex Hessian equations,} Anal. PDE {\bf 9} (2016), no.~7, 1693--1709.

\bibitem{PPZ16} D. H. Phong, S. Picard and X. Zhang, {\it The Fu-Yau equation with negative slope parameter,} preprint. arXiv: 1602.08838. To appear in Invent. Math.


\bibitem{plis} S. Pli\'s, {\it The smoothing of m-subharmonic functions,} preprint. arXiv: 1312.1906.





\bibitem{sem92} S. Semmes, {\it Complex Monge-Amp\`ere and symplectic manifolds,} Amer. J. Math. {\bf 114} (1992), no.~3, 495--550.





\bibitem{sprucknote}J. Spruck, {\it Geometric aspects of the theory of fully nonlinear elliptic equations,} in  Global theory of minimal surfaces, 283--309, Clay Math. Proc., 2, Amer. Math. Soc., Providence, RI, 2005.






\bibitem{szekelyhidi15} G. Sz\'ekelyhidi, {\it Fully non-linear elliptic equations on compact Hermitian manifolds.} preprint,  arXiv:1501.02762v3. To appear in J. Differential Geom.

\bibitem{SzTW} G. Sz\'ekelyhidi, \ V. Tosatti and \ B. Weinkove, Gauduchon metrics with prescribed volume form, preprint,  arXiv:1503.04491








\bibitem{twwy14} V. Tosatti, Y. Wang, B. Weinkove and X. Yang, {\it $C\sp {2,\alpha}$ estimates for nonlinear elliptic equations in complex and almost complex geometry.} Calc. Var. Partial Differential Equations {\bf 54} (2015), no. 1, 431--453.











\bibitem{TW13a} V. Tosatti and B. Weinkove,  {\it The Monge-Amp\`ere equation for $(n-1)$-plurisubharmonic functions on a compact K\"ahler manifold,}  J. Amer. Math. Soc. {\bf 30} (2017), no.~2, 311--346. 

\bibitem{TW13b} V. Tosatti and B. Weinkove, {\it Hermitian metrics, $(n-1, n-1)$ forms and Monge-Amp\`ere equations.} preprint, arXiv:1310.6326. To appear in J. Reine Angew. Math. 2017.




\bibitem{wang96} X. J. Wang, {\it A class of fully nonlinear elliptic equations and related functionals,} Indiana Univ. Math. J. {\bf 43} (1994), no.~1, 25--54. 



\bibitem{ywang} Y. Wang,  {\it A viscosity approach to the Dirichlet problem for complex Monge-Amp\`ere equations.} Math. Z. {\bf 272} (2012), no. 1-2, 497--513.



\bibitem{dzhang15} D. Zhang, {\it Hessian equations on closed Hermitian manifolds.} preprint, arXiv:1501.03553.

\end{thebibliography}
\end{document}